\documentclass[preprint,times]{article}
\usepackage{amsmath,mathrsfs,amssymb,amsthm,amscd}
\usepackage[]{fontenc}
\usepackage{xy}
\usepackage[hyperindex]{hyperref}
\hypersetup{breaklinks=true}
\usepackage{graphicx,color}
\usepackage{ifpdf}
\xyoption{all}

\numberwithin{equation}{section}
\nopagebreak%
\theoremstyle{plain}
\newtheorem{theorem}[equation]{Theorem}
\newtheorem{corollary}[equation]{Corollary}
\newtheorem{proposition}[equation]{Proposition}
\newtheorem{lemma}[equation]{Lemma}

\newtheorem{theorem-definition}[equation]{Theorem Definition}

\theoremstyle{definition}
\newtheorem{definition}[equation]{Definition}
\newtheorem{notation}[equation]{Notation}
\newtheorem{example}[equation]{Example}

\newtheorem{remark}[equation]{Remark}

\DeclareMathOperator{\Gr}{Gr}

\DeclareMathOperator{\Fil}{Fil}

\DeclareMathOperator{\ocone}{\overline{cone}}

\DeclareMathOperator{\Td}{Td}
\DeclareMathOperator{\cht}{\widetilde{ch}}

\DeclareMathOperator{\Tot}{Tot}

 \DeclareMathOperator{\dd}{d}

\DeclareMathOperator{\ch}{ch} 
\DeclareMathOperator{\rk}{rk} 
 
\DeclareMathOperator{\Ob}{Ob}

\DeclareMathOperator{\Hom}{Hom} 
\DeclareMathOperator{\uHom}{\underline{Hom}}
 
\DeclareMathOperator{\Id}{id}

\DeclareMathOperator{\Spec}{Spec}

\DeclareMathOperator{\cone}{cone}

\DeclareMathOperator{\Cb}{\mathbf{C}^{b}}
\DeclareMathOperator{\KA}{\mathbf{KA}}
\DeclareMathOperator{\oV}{\overline {\mathbf{V}}^{b}}
\DeclareMathOperator{\oVo}{\overline {\mathbf{V}}^{0}}
\DeclareMathOperator{\Vb}{\mathbf{V}^{b}}

\DeclareMathOperator{\Db}{\mathbf{D}^{b}}

\DeclareMathOperator{\oDb}{\overline{\mathbf{D}}^{b}}
\DeclareMathOperator{\oSm}{\overline{\mathbf{Sm}}}
\DeclareMathOperator{\Sm}{\mathbf{Sm}}
\DeclareMathOperator{\ocF}{\overline{\mathcal{F}}}
\DeclareMathOperator{\Coh}{\mathbf{Coh}}

\DeclareMathOperator{\even}{even}
\DeclareMathOperator{\odd}{odd}

\DeclareMathOperator{\im}{Im\,}
\renewcommand{\Im}{{\im}}

\newcommand{\ov}{\overline}
\newcommand{\dra}{\dashrightarrow}
\newcommand{\Ld}{}
\newcommand{\Rd}{}
\newcommand{\RuHom}{\Rd\uHom}
\newcommand{\Lotimes}{\otimes^{\Ld}}

\newcommand{\an}{\text{{\rm an}}}

\newcommand{\ZZ}{{\mathbb Z}}
\newcommand{\RR}{{\mathbb R}}

\newcommand{\CC}{{\mathbb C}}

\newcommand{\PP}{{\mathbb P}}
\newcommand{\DD}{{\mathbb D}}

\def\?{\ ???\ \immediate\write16{}%
\immediate\write16{Warning: There was still a question mark . . . }%
\immediate\write16{}}

\begin{document}

\title{Hermitian
  structures on the derived category of coherent sheaves}
\author{
\\
\\
	Jos\'e I. Burgos Gil\footnote{Partially supported by grant
          MTM2009-14163-C02-01 and CSIC research project 2009501001.}\\ 
	\small{Instituto de Ciencias Matem\'aticas (ICMAT-CSIC-UAM-UCM-UC3)}\\
	\small{Consejo Superior de Investigaciones Cient\'ificas (CSIC)}\\
	\small{Spain}\\
	\small{\texttt{burgos@icmat.es}}\\
	\and
	Gerard Freixas i Montplet\footnote{Partially supported by
          grant MTM2009-14163-C02-01}\\
	\small{Institut de Math\'ematiques de Jussieu (IMJ)}\\
	\small{Centre National de la Recherche Scientifique (CNRS)}\\
	\small{France}\\
	\small{\texttt{freixas@math.jussieu.fr}}\\
	\and
	R\u azvan Li\c tcanu\footnote{Supported by CNCSIS -UEFISCSU,
          project number PNII - IDEI 2228/2008 .}\\ 
	\small{Faculty of Mathematics}\\
	\small{University Al. I. Cuza Iasi}\\
	\small{Romania}\\
	\small{\texttt{litcanu@uaic.ro}}
}

\maketitle
\begin{abstract}
  The main objective of the present paper is to set up the theoretical
  basis and the language needed to deal with the problem of direct
  images of hermitian vector bundles for projective non-necessarily
  smooth morphisms. To this end, we first define hermitian structures on the
  objects of the bounded derived category of coherent sheaves on a
  smooth complex variety. Secondly we extend the theory of Bott-Chern
  classes to these hermitian structures. Finally we introduce the
  category $\oSm_{\ast/\CC}$ whose morphisms are projective morphisms
  with a hermitian structure on the relative tangent complex.
\end{abstract}

\section{Introduction}
\label{sec:introduction}

Derived categories were introduced in the 60's of the last century by
Grothendieck and Verdier in order to study and generalize duality
phenomenons in Algebraic Geometry (see \cite{Hartshorne:rd},
\cite{Verdier:MR1453167}). Since
then, derived categories had become a standard tool in Algebra and Geometry
and the right framework to define derived functors and to study
homological properties. A paradigmatic example is the definition of
direct image of sheaves. Given a map $\pi\colon X\to Y$ between
varieties and a sheaf $\mathcal{F}$ on $X$, there is a notion of
direct image $\pi _{\ast}\mathcal{F}$. We are not specifying what kind
of variety or sheaf we are talking about because the same circle of
ideas can be used in many different settings. This direct image is not
exact in the sense that if $f\colon \mathcal{F}\to \mathcal{G}$ is a
surjective map of sheaves, the induced morphism $\pi _{\ast}f\colon
\pi _{\ast}\mathcal{F}\to \pi _{\ast}\mathcal{G}$ is not necessarily
surjective. One then can define a \emph{derived} functor $R\pi
_{\ast}$ that takes values in the derived category of sheaves on $Y$
and that is exact in an appropriate sense. This functor encodes a lot
of information about the topology of the fibres of the map $\pi $.

The interest for the derived category of coherent sheaves on a variety
exploded with the celebrated 1994 lecture by Kontsevich
\cite{Kontsevich:MR1403918}, interpreting
mirror symmetry as an equivalence between the derived category of the
Fukaya category of certain symplectic manifold and the derived category of
coherent sheaves of a dual complex manifold. In the last decades, many
interesting results about the derived category of coherent sheaves
have been obtained, like Bondal-Orlov
Theorem \cite{BondalOrlov:MR1818984} that shows that a projective
variety with ample canonical or anti-canonical bundle can be recovered
from its derived category of coherent sheaves. Moreover, new tools for
studying algebraic varieties have been developed in the context of
derived categories like the
Fourier-Mukai transform \cite{Mukai:FT}. The interested reader is
referred to books like \cite{Huybrechts:FM} and
\cite{Bartoccietal:MR2511017} for a thorough exposition of recent
developments in this area.

Hermitian vector bundles are ubiquitous in Mathematics. An
interesting problem is to define the direct image of hermitian vector
bundles. More concretely, let $\pi \colon X\to Y$ be a proper
holomorphic map of complex manifolds and let $\ov E=(E,h)$ be a
hermitian holomorphic vector bundle on $X$. We would like to define
the direct image
$\pi_{\ast}\ov E$ as something as close as possible to a hermitian
vector bundle on $Y$. The information that would be easier to extract from
such a direct image 
is encoded in the determinant of the cohomology
\cite{Deligne:dc}, that can be defined directly. Assume that $\pi $ is a 
submersion and that we have chosen a hermitian metric on the relative
tangent bundle $T_{\pi }$ of $\pi $ satisfying certain technical
conditions. Then the determinant line bundle $\lambda
(E)=\det(R\pi _{\ast}E)$ can be equipped  with the Quillen metric
(\cite{Quillen:dCRo}, \cite{BismutFreed:EFI}, \cite{BismutFreed:EFII}), that
depends on the metrics on $E$ and $T_{\pi }$ and is constructed using
the analytic torsion \cite{RaySinger:ATCM}. The Quillen metric has
applications in Arithmetic Geometry (\cite{Faltings:cas}, \cite{Deligne:dc},
\cite{GilletSoule:aRRt}) and also in String Theory
(\cite{Yau:MR915812}, \cite{AlvarezGaumeetals:MR908551}). Assume
furthermore that the higher direct image sheaves $R^{i}\pi _{\ast}E$
are locally free. In general it is not possible to define an
analogue of the Quillen metric as a hermitian metric on each vector
bundle $R^{i}\pi
_{\ast}E$. But following Bismut and K\"ohler \cite{Bismut-Kohler}, one
can do something almost as good. We can define the $L^{2}$-metric on $R^{i}\pi
_{\ast}E$ and \emph{correct} it using the higher analytic torsion
form. Although this \emph{corrected metric} is not properly a
hermitian metric, it is enough for constructing
characteristic forms and it appears in the Arithmetic
Grothendieck-Riemann-Roch
Theorem in higher degrees~\cite{GilletRoesslerSoule:_arith_rieman_roch_theor_in_higher_degrees_p}.

The main objective of the present paper is to set up the theoretical
basis and the language needed to deal with the problem of direct
images of hermitian vector bundles for projective non-necessarily
smooth morphisms.
This program will be continued in the
subsequent paper \cite{BurgosFreixasLitcanu:GenAnTor} where we
give an axiomatic characterization of analytic torsion forms and we
generalize them to projective morphisms.
The ultimate goal of this program is to state and prove an Arithmetic
Grothendieck-Riemann-Roch Theorem for general projective
morphisms. This last result will be the topic of a forthcoming paper.

When dealing with direct images of hermitian vector bundles for non
smooth morphisms, one is naturally
led to consider hermitian structures on objects of the bounded derived
category of coherent sheaves $\Db$. One reason for this is that, for
a non-smooth projective morphism $\pi $, instead
of the relative tangent
bundle one should consider the relative tangent complex, that defines
an object of $\Db(X)$. Another reason is that, in general, the higher
direct images $R^{i}\pi _{\ast}E$ are coherent sheaves and the derived
direct image $R\pi _{\ast}E$ is an object of $\Db(Y)$.

Thus the first goal of this paper is to define hermitian structures.
A possible starting point is to
define a hermitian metric on an object
$\mathcal{F}$ of $\Db(X)$ as an isomorphism $E\dra \mathcal{F}$ in
$\Db(X)$, with $E$ a bounded complex of vector bundles, together
with a choice of a hermitian metric on each constituent vector bundle
of $E$.
Here we find a problem, because even being $X$ smooth,
in the bounded derived category of
coherent sheaves of $X$, not every object can be
represented by a bounded complex of locally free sheaves (see
\cite{Voisin:counterHcK} and Remark \ref{rem:3}). Thus the previous
idea does not work for general complex manifolds.
To avoid this
problem we will restrict ourselves to the
algebraic category. Thus, from now on the letters $X,\ Y,\dots$ will
denote smooth algebraic varieties over $\CC$, and all sheaves
will be algebraic.

With the previous definition of hermitian metric, for each object of
$\Db(X)$ we obtain a class of
metrics that is too wide. Different constructions
that ought to produce the same metric produce in fact different
metrics. This indicates that we may define a hermitian structure as an equivalence
class of hermitian metrics.

Let us be more precise. Being $\Db(X)$ a triangulated category, to
every morphism
$\mathcal{F}\overset{f}{\dra}\mathcal{G}$ in $\Db(X)$ we can associate
its cone, that
is defined up to a (not unique) isomorphism by the fact that
\begin{displaymath}
  \mathcal{F}\dra \mathcal{G}\dra \cone(f)\dra \mathcal{F}[1]
\end{displaymath}
is a distinguished triangle. If now $\mathcal{F}$ and $\mathcal{G}$
are provided with hermitian metrics, we want that $\cone(f)$ has an
induced hermitian structure that is well defined up to
\emph{isometry}. By choosing a representative of the map $f$ by means
of morphisms of complexes of vector bundles, we can induce a hermitian metric on
$\cone(f)$, but this hermitian metric depends on the choices. The idea
behind the definition of hermitian structures is to introduce the
finest equivalence relation between metrics such that all possible
induced hermitian metrics on $\cone(f)$ are equivalent.

Once we have defined hermitian structures a new invariant of $X$
can be naturally defined. Namely, the set of hermitian structures on a zero object of
$\Db(X)$ is an abelian group that we denote $\KA(X)$ (Definition
\ref{def:KA}). In the same way that $K_{0}(X)$ is the universal
abelian group for additive characteristic classes of vector bundles,
$\KA(X)$ is the universal abelian group for secondary characteristic
classes of acyclic complexes of hermitian vector bundles (Theorem~\ref{thm:8}).

Secondary characteristic classes constitute other of the central topics of
this paper.  Recall that to each vector bundle we can associate its
Chern character, that is an additive characteristic class. If the
vector bundle is provided with a hermitian metric, we can use
Chern-Weil theory to construct a
concrete representative of the Chern character, that is a differential
form. This characteristic
form is additive only for orthogonally split short exact sequences and
not for general short exact sequences. Bott-Chern classes were
introduced in \cite{BottChern:hvb} and are secondary classes that
measure the lack of additivity of the characteristic forms.

The Bott-Chern classes have been extensively used in Arakelov Geometry
(\cite{GilletSoule:vbhm}, \cite{BismutGilletSoule:at}) and they can be
used to construct characteristic classes in higher $K$-theory
(\cite{Burgos-Wang}).
The second goal of this paper is to
extend the definition of
additive Bott-Chern
classes to the derived category. This is the most general definition
of additive Bott-Chern classes and encompasses both, the Bott-Chern
classes defined in \cite{BismutGilletSoule:at} and the ones defined in
\cite{Ma:MR1765553} (Example \ref{exm:1}).

Finally, recall that the hermitian structure on the direct image of a
hermitian vector bundle should also depend on a hermitian structure on
the relative tangent complex. Thus the last goal of this paper is to
introduce the category  $\ov{\Sm}_{\ast/\CC}$ (Definition
\ref{def:16}, Theorem \ref{thm:17}). The objects of this category are smooth
algebraic varieties over $\CC$ and the morphisms are pairs
$\ov{f}=(f,\ov{T}_{f})$ formed by a
projective morphism of smooth complex varieties $f$, together with a
hermitian structure on the relative tangent complex $T_{f}$. The main
difficulty here is to define the composition of two such morphisms.
 The remarkable
fact is that the hermitian cone construction enables us to define a
composition
rule for these morphisms.

We describe with more detail the contents of each section.

In Section \ref{sec:meager-complexes} we
define and characterize the notion of \textit{meager complex}
(Definition \ref{def:9} and Theorem \ref{thm:3}). Roughly speaking,
meager complexes are bounded acyclic complexes of hermitian vector
bundles whose Bott-Chern classes vanish for structural
reasons. We then introduce the concept of tight morphism (Definition
\ref{def:tight_morphism}) and tight equivalence relation (Definition
\ref{def:17}) between bounded complexes of hermitian vector
bundles. We explain a series of useful computational rules on the
monoid of hermitian vector bundles modulo tight equivalence relation,
that we call \textit{acyclic calculus} (Theorem \ref{thm:7}). We prove
that the submonoid of acyclic complexes modulo meager
complexes has a structure of abelian group, this is the group
$\KA(X)$ mentioned previously.

With these tools at hand, in Section
\ref{sec:oDb} we define hermitian
structures on objects of $\Db(X)$ and we introduce the category
$\oDb(X)$.  The objects of the category
$\oDb(X)$ are objects of $\Db(X)$ together with a hermitian structure,
and the morphisms are just morphisms in $\Db(X)$. Theorem \ref{thm:13}
is devoted to describe the structure of the forgetful functor
$\oDb(X)\to\Db(X)$. In particular, we show that the group $\KA(X)$
acts on the fibers of this functor, freely and
transitively.

An important example of use of hermitian structures is
the construction of the \textit{hermitian cone} of a morphism in
$\oDb(X)$ (Definition \ref{def:her_cone}), which is well defined only
up to tight isomorphism. We also study several elementary
constructions in $\oDb(X)$. Here we mention the classes of
isomorphisms and distinguished triangles in $\oDb(X)$. These classes
lie in the group $\KA(X)$ and their properties are listed in Theorem
\ref{thm:10}. As an application we show that $\KA(X)$ receives classes
from $K_{1}(X)$ (Proposition \ref{prop:K1_to_KA}).

Section \ref{sec:bott-chern-classes} is devoted to the extension of
Bott-Chern classes to the derived category. For every additive
genus, we associate to each isomorphism or
distinguished triangle in $\oDb(X)$ a Bott-Chern class satisfying
properties analogous to the classical ones.

We conclude the paper with Section \ref{sec:multiplicative-genera},
where we extend the definition of Bott-Chern classes to
multiplicative genera and in particular to the Todd genus. In this
section
we also define the category $\ov{\Sm}_{\ast/\CC}$.

\textbf{Acknowledgements:}
We would like to thank the following institutions where part of the
research conducting to this paper was done. The CRM in Bellaterra
(Spain), the CIRM in Luminy (France), the Morningside Institute of Beijing
(China), the University of Barcelona and the IMUB, the Alexandru Ioan
Cuza University of Iasi,
the Institut de Math\'ematiques de Jussieu and the ICMAT (Madrid).

We would also like to thank D. Eriksson and D. R\"ossler for several
discussions on the subject of this paper.

\section{Meager complexes and acyclic calculus}
\label{sec:meager-complexes}

The aim of this section is to construct a universal group for
additive Bott-Chern  classes of acyclic complexes of
hermitian vector bundles. To this end we first introduce and study the
class of meager complexes. Any Bott-Chern class that is
additive for certain short exact sequences of acyclic complexes (see
\ref{thm:8}) and that
vanishes on orthogonally split complexes, necessarily vanishes on
meager complexes. Then we develop an acyclic calculus that will ease
the task to check if a particular complex is meager. Finally we
introduce the group $\KA$, which is the universal group for additive
Bott-Chern classes.

Let $X$ be a complex algebraic variety over $\CC$, namely a reduced
and separated scheme of finite type over $\CC$. We denote by $\Vb(X)$
the exact category of bounded complexes of algebraic vector bundles on
$X$. Assume in addition that $X$ is smooth over $\CC$. Then $\oV(X)$
is defined as the category of pairs $\ov E=(E,h)$, where $E\in \Ob
\Vb(X)$ and $h$ is a smooth hermitian metric on the complex of
analytic vector bundle $E^{\an}$. From now on we shall make no
distinction between $E$ and $E^{\an}$. The complex $E$ will be called
\emph{the underlying complex of } $\ov E$.  We will denote by the
symbol $\sim$ the quasi-isomorphisms in any of the above categories.

A basic construction in $\Vb(X)$ is the cone of a morphism of
complexes. Recall that, if $f\colon E\to F$ is such a morphism, then,
as a graded vector bundle $\cone(f)=E[1]\oplus F$ and the differential
is given by $\dd(x,y)=(-\dd x,f(x)+\dd y).$
We can extend the cone construction easily to $\oV(X)$ as follows.

\begin{definition} \label{def:14}
If $f\colon \overline E\to \overline F$ is a morphism in $\oV(X)$,
\emph{the hermitian cone} of $f$,
denoted by $\ocone(f)$, is defined as the cone of $f$
provided with the orthogonal sum hermitian metric.

When the morphism
is clear from the context we will sometimes
denote $\ocone(f)$ by $\ocone(\ov E,\overline F)$.
\end{definition}

\begin{remark} \label{rem:2} Let $f\colon \ov E\to \ov F$ be a
  morphism in $\oV(X)$. Then
  there is an exact sequence of complexes
  \begin{displaymath}
    0\longrightarrow \overline F\longrightarrow \ocone(f)
    \longrightarrow \overline E[1]\longrightarrow 0,
  \end{displaymath}
  whose constituent short exact sequences are orthogonally
  split. Conversely, if
  \begin{displaymath}
    0\longrightarrow \overline F\longrightarrow \ov G
    \longrightarrow \overline E[1]\longrightarrow 0
  \end{displaymath}
  is a short exact sequence all whose constituent exact sequences are
  orthogonally split, then there is a natural section $s\colon E[1]\to
  G$. The image of $\dd s-s\dd$ belongs to
  $F$ and, in fact, determines a morphism
  of complexes
  \begin{displaymath}
    f_{s}:=\dd s-s\dd\colon \ov E
    \longrightarrow \ov F.
  \end{displaymath}
  Moreover, there is a natural isometry $\overline G \cong \ocone
  (f_{s})$.
\end{remark}

The hermitian cone has the following useful property.

  \begin{lemma}\label{lemm:13}
    Consider a diagram in $\oV(X)$
    \begin{displaymath}
      \xymatrix{
        \overline{E}'\ar[r]^{f'}\ar[d]_{g'} &\overline{F}'\ar[d]^{g}\\
        \overline{E}\ar[r]^{f}	&\overline{F}.
      }
    \end{displaymath}
    Assume that the diagram is commutative up to homotopy and fix a
    homotopy $h$. The homotopy $h$ induces
    morphisms of complexes
    \begin{align*}
      &\psi \colon\ocone(f')\longrightarrow\ocone(f)\\
      &\phi \colon\ocone(-g')\longrightarrow\ocone(g)
    \end{align*}
    and there is a natural isometry of complexes
    \begin{displaymath}
      \ocone(\phi )\overset{\sim}{\longrightarrow}\ocone(\psi ).
    \end{displaymath}
    Morever, let $h'$ be a second homotopy between $g\circ f'$ and
    $f\circ g'$ and let $\psi '$ be the induced morphism. If there
    exists a higher homotopy between $h$ and $h'$, then $\psi
    $ and $\psi '$ are homotopically equivalent.
  \end{lemma}
  \begin{proof}
    Since $h\colon E'\to F[-1]$ is a homotopy between $gf'$ and $fg'$,
    we have
    \begin{equation}\label{eq:homotopy_square_1}
      gf'-fg'=\dd h+h\dd.
    \end{equation}
    First of all, define the arrow $\psi \colon\ocone(f')\to\ocone(f)$ by
    the following rule:
    \begin{displaymath}
      \psi (x',y')=(g'(x'),g(y')+h(x')).
    \end{displaymath}
    From the definition of the differential of a cone and the homotopy
    relation (\ref{eq:homotopy_square_1}), one easily checks that $\psi $
    is a morphism of complexes. Now apply the same construction to the
    diagram
    \begin{equation}\label{eq:homotopy_square_2}
      \xymatrix{
        \overline{E}'\ar[r]^{-g'}\ar[d]_{-f'}	&\overline{E}\ar[d]^{f}\\
        \overline{F'}\ar[r]^{g}	&\overline{F}.
      }
    \end{equation}
    The diagram (\ref{eq:homotopy_square_2}) is still commutative up
    to homotopy and $h$ provides such a homotopy. We obtain a
    morphism of complexes $\phi :\ocone(-g')\to\ocone(g)$, defined by the
    rule
    \begin{displaymath}
      \phi (x',x)=(-f'(x'),f(x)+h(x')).
    \end{displaymath}
    One easily checks that a suitable reordering of factors sets an
    isometry of complexes between $\ocone(\phi )$ and $\ocone(\psi )$.
    Assume now that $h'$ is a second homotopy and that there is a
    higher homotopy $s\colon \ov E' \to \ov F[-2]$ such that
    \begin{displaymath}
      h'-h=\dd s- s \dd.
    \end{displaymath}
    Let $H\colon \ocone(f')\to \ocone(f)[-1]$ be given by
    $H(x',y')=(0,s(x'))$. Then
    \begin{displaymath}
      \psi '-\psi =\dd H + H \dd.
    \end{displaymath}
    Hence $\psi $ and $\psi' $ are homotopically equivalent.
  \end{proof}

  Recall that, given a morphism of complexes $f\colon \ov E \to \ov
  F$, we use the abuse of notation $\ocone(f)=\ocone(\ov E, \ov
  F)$. As seen in the previous lemma, sometimes it is natural to consider
  $\ocone(-f)$. With the notation above it will be denoted also by
  $\ocone(\ov E,\ov F)$. Note that this ambiguity is harmless because
  there is a natural isometry between $\ocone(f)$ and $\ocone(-f)$. Of
  course, when more than one morphism between $\ov E$ and $\ov F$ is
  considered, the above notation should be avoided.

  With this convention, Lemma \ref{lemm:13} can
  be written as
  \begin{equation}
    \label{eq:53}
    \ocone(\ocone(\ov E',\ov E),\ocone(\ov F',\ov F))\cong
    \ocone(\ocone(\ov E',\ov F'),\ocone(\ov E,\ov F)).
  \end{equation}

\begin{definition}\label{def:11}
  We will denote by $\mathscr{M}_{0}=\mathscr{M}_{0}(X)$ the subclass of
  $\oV(X)$ consisting of
  \begin{enumerate}
  \item the orthogonally split complexes;
  \item all objects $\ov E$ such that there
  is an acyclic complex $\ov F$ of $\oV(X)$, and an isometry $\ov E
  \to \ov F\oplus \ov F[1]$.
  \end{enumerate}
\end{definition}

We want to stabilize $\mathscr{M}_{0}$ with respect to hermitian cones.

\begin{definition}\label{def:9}
  We will denote by $\mathscr{M}=\mathscr{M}(X)$ the smallest subclass of
  $\oV(X)$ that satisfies the following
  properties:
  \begin{enumerate}
  \item \label{item:12} it contains $\mathscr{M}_{0}$;
  \item \label{item:13} if $f\colon \overline
  E\to \overline F$ is a morphism and two
  of $\overline E$, $\overline F$ and $\ocone(f)$ belong
  to $\mathscr{M}$, then so does the third.
  \end{enumerate}
  The elements of $\mathscr{M}(X)$ will be called \emph{meager
    complexes}.
\end{definition}

We next give a characterization of meager complexes.
For this, we introduce two auxiliary classes.
\begin{definition} \label{def:12}
  \begin{enumerate}
  \item \label{item:29} Let $\mathscr{M}_{F}$ be the subclass of
    $\oV(X)$ that contains all complexes $\overline E$ that have a
    finite filtration $\Fil$ such that
    \begin{enumerate}
    \item[({\bf A})] \label{item:14} for every $p,n\in \ZZ$, the exact
      sequences
      \begin{displaymath}
        0\to \Fil^{p+1}\overline E^{n}\to \Fil^{p}\overline E^{n}\to
        \Gr^{p}_{\Fil}\overline E^{n}\to 0,
      \end{displaymath}
      with the induced metrics, are orthogonally split short exact
      sequences of vector bundles;
    \item[({\bf B})] \label{item:15} the complexes
      $\Gr^{\bullet}_{\Fil}\overline E$ belong to $\mathscr{M}_{0}$.
    \end{enumerate}
  \item \label{item:32} Let $\mathscr{M}_{S}$ be the subclass of
    $\oV(X)$ that contains all complexes $\overline E$ such that there
    is a morphism of complexes $f\colon \overline E\to \overline F$
    and both $\overline F$ and $\ocone (f)$ belong to
    $\mathscr{M}_{F}$.
  \end{enumerate}
\end{definition}

  \begin{lemma} \label{lemm:10} Let $0\to \ov E\to \ov F\to \ov G\to
    0$ be an exact sequence in $\oV(X)$ whose
    constituent rows are orthogonally split. Assume $\ov E$ and $\ov
    G$ are in $\mathscr{M}_{F}$ . Then $\ov F\in\mathscr{M}_{F}$. In
    particular, $\mathscr{M}_{F}$ is closed under cone formation.
  \end{lemma}
  \begin{proof}
    For the first claim, notice that the filtrations of $\ov E$ and $\ov G$
    induce a filtration on $\ov F$ satisfying conditions
    \ref{def:12}~({\bf A}) and \ref{def:12}~({\bf B}). The second
    claim then follows by Remark \ref{rem:2}.
  \end{proof}

\begin{example} \label{exm:2}
  Given any complex $\ov E\in \Ob \oV(X)$, the complex
  $\ocone(\Id_{\ov E})$ belongs to
  $\mathscr{M}_{F}$. This can be seen by induction on the length of
  $\ov E$ using Lemma \ref{lemm:10} and the b\^ete filtration of $\ov
  E$. For the starting point of the induction one takes into account
  that, if $\ov E$ has only one non zero
  degree, then $\ocone(\Id_{\ov E})$ is orthogonally split. In fact,
  this argument shows something slightly stronger. Namely, the complex
  $\ocone(\Id_{\ov E})$ admits a finite filtration $\Fil$ satisfying
  \ref{def:12}~({\bf A}) and such that the complexes
  $\Gr^{\bullet}_{\Fil}  \ocone(\Id_{\ov E})$ are orthogonally
  split.
\end{example}

\begin{theorem}\label{thm:3} The equality
  \begin{math}
    \mathscr{M}=\mathscr{M}_{S}
  \end{math}
  holds.
\end{theorem}
\begin{proof}
  We start by proving that $\mathscr{M}_{F}\subset \mathscr{M}$.
  Let $\overline E\in \mathscr{M}_{F}$ and let $\Fil$ be any
  filtration that satisfies conditions \ref{def:12}~({\bf A}) and
  \ref{def:12}~({\bf B}).  We show that $
  \overline E\in \mathscr{M}$ by induction
  on the length of $\Fil$. If
  $\Fil $ has length one, then $\overline E$ belongs to
  $\mathscr{M}_{0}\subset \mathscr{M}$. If the length of $\Fil$ is
  $k>1$, let $p$ be
  such that $\Fil^{p}\overline E=\overline E$ and
  $\Fil^{p+1}\overline E\not = \overline E$. On the one hand, $\Gr
  ^{p}_{\Fil}\overline E [-1]\in
  \mathscr{M}_{0}\subset \mathscr{M}$ and, on the other hand, the
  filtration $\Fil$
  induces a filtration on $\Fil^{p+1}\overline E$ fulfilling
  conditions \ref{def:12}~({\bf A}) and \ref{def:12}~({\bf B}) and has
  length $k-1$. Thus, by induction hypothesis, $\Fil^{p+1}\overline
  E\in \mathscr{M}$. Then, by
  Lemma \ref{lemm:10}, we deduce that $\overline E\in
  \mathscr{M}$.

  Clearly, the fact that $\mathscr{M}_{F}\subset \mathscr{M}$
  implies that $\mathscr{M}_{S}\subset \mathscr{M}$. Thus, to
  prove the theorem, it only remains to show that $\mathscr{M}_{S}$
  satisfies the condition \ref{def:9}~\ref{item:13}.

  The content of the next result is that
  the apparent asymmetry in the definition of
  $\mathscr{M}_{S}$ is not real.

  \begin{lemma}\label{lemm:9}
    Let $\overline E\in \Ob \oV(X)$. Then there is a morphism $f\colon
    \overline
    E\to \overline F$ with $\overline F$ and
    $\ocone(f)$ in $\mathscr{M}_{F}$ if and only if there is a
    morphism $g\colon \overline G\to \overline E$ with
    $\overline G$ and
    $\ocone(g)$ in $\mathscr{M}_{F}$.
  \end{lemma}
  \begin{proof}
    Assume that there is a morphism $f\colon \overline
    E\to \overline F$ with $\overline F$ and
    $\ocone(f)$ in $\mathscr{M}_{F}$. Then, write $\overline
    G= \ocone(f)[-1]$ and let $g\colon \overline G\to
    \overline E$ be the natural map. By hypothesis, $\overline
    G\in \mathscr{M}_{F}$. Moreover, since there is a natural isometry
    \begin{displaymath}
      \ocone(\ocone(\overline E,\overline
      F)[-1],\overline E)\cong
      \ocone(\ocone(\Id_{\overline E})[-1],\overline
      F),
    \end{displaymath}
    by Example \ref{exm:2} and Lemma \ref{lemm:10} we obtain that
    $\ocone(g)\in \mathscr{M}_{F}$. Thus we have proved one
    implication. The proof of the other implication is analogous.
   \end{proof}

  Let now $f\colon \overline E\to \overline F$ be a morphism of
  complexes with $\overline E, \overline F\in
  \mathscr{M}_{S}$. We want to show that $\ocone(f)\in
  \mathscr{M}_{S}$. By Lemma \ref{lemm:9}, there are morphisms of
  complexes
  $g\colon \overline G\to \overline E$ and
  $h\colon \overline H\to \overline F$ with $\overline
  G,\ \overline H,\ \ocone(g),\ \ocone(h)\in
  \mathscr{M}_{F}$. We consider the map $\overline G\to
  \ocone(h)$ induced by $f\circ g$. Then we write
  \begin{displaymath}
    \overline{G'}=\ocone(\overline G, \ocone(h) )[-1].
  \end{displaymath}
  By Lemma \ref{lemm:10}, we have that $\overline{G'}\in
  \mathscr{M}_{F}$.  We denote by $g'\colon G'\to E$ and $k\colon G'\to H$ the
  maps $g'(a,b,c)=g(a)$ and $k(a,b,c)=-b$.

  There is an exact sequence
  \begin{displaymath}
    0\to \ocone(h)\to \ocone(g')\to \ocone(g)\to 0
  \end{displaymath}
  whose constituent short exact sequences are orthogonally split.
  Since $\ocone(h)$ and $\ocone(g)$ belong to $\mathscr{M}_{F}$,
  Lemma \ref{lemm:10} insures that $\ocone(g')$ belongs to
  $\mathscr{M}_{F}$ as well.

  There is a diagram
  \begin{equation}\label{eq:51}
    \xymatrix{
      \overline {G'} \ar[d]_{g'} \ar[r]^{k} & \overline H\ar[d]^{h} \\
      \overline E \ar[r]^{f} & \overline F
    }
  \end{equation}
  that commutes up to homotopy. We fix the homotopy $s\colon \ov G'\to
  F$ given by $s(a,b,c)=c$. By Lemma \ref{lemm:13} there is
  a natural isometry
  \begin{displaymath}
    \ocone(\ocone(g'),\ocone(h))\cong \ocone(\ocone(-k),\ocone(f)).
  \end{displaymath}
  Applying Lemma \ref{lemm:10} again, we have that $\ocone(-k)$ and
  $\ocone(\ocone(g'),\ocone(h))$ belong to
  $\mathscr{M}_{F}$. Therefore $\ocone(f)$ belongs to
  $\mathscr{M}_{S}$.

  \begin{lemma} \label{lemm:11}
    Let $f\colon \ov E\to \ov F$ be a morphism in $\oV(X)$.
    \begin{enumerate}
    \item \label{item:18} If $\ov E\in \mathscr{M}_{S}$ and $\ocone(f)\in
      \mathscr{M}_{F}$ then $\ov F\in \mathscr{M}_{S}$.
    \item \label{item:19} If $\ov F\in \mathscr{M}_{S}$ and $\ocone(f)\in
      \mathscr{M}_{F}$ then $\ov E\in \mathscr{M}_{S}$.
    \end{enumerate}
  \end{lemma}
  \begin{proof}
    Assume that $\ov E\in \mathscr{M}_{S}$ and $\ocone(f)\in
      \mathscr{M}_{F}$. Let $g\colon \ov G\to \ov E$ with $\ov G\in
      \mathscr{M}_{F}$ and $\ocone(g)\in \mathscr{M}_{F}$. By Lemma
      \ref{lemm:10} and Example \ref{exm:2}, $\ocone(\ocone(\Id_{\ov
        G}),\ocone(f))\in \mathscr{M}_{F}$. But 
      there is a natural isometry of complexes
      \begin{displaymath}
        \ocone(\ocone(\Id_{\ov G}),\ocone(f))\cong
        \ocone(\ocone(\ocone(g)[-1],\ov G),\ov F).
      \end{displaymath}
      Since, by Lemma \ref{lemm:10}, $\ocone(\ocone(g)[-1],\ov G)\in
      \mathscr{M}_{F}$, then $\ov F\in \mathscr{M}_{S}$.

      The second statement of the lemma is proved using the dual
      argument.
  \end{proof}

  \begin{lemma} \label{lemm:12}
    Let $f\colon \ov E\to \ov F$ be a morphism in $\oV(X)$.
    \begin{enumerate}
    \item \label{item:20} If $\ov E\in \mathscr{M}_{F}$ and $\ocone(f)\in
      \mathscr{M}_{S}$ then $\ov F\in \mathscr{M}_{S}$.
    \item \label{item:21} If $\ov F\in \mathscr{M}_{F}$ and $\ocone(f)\in
      \mathscr{M}_{S}$ then $\ov E\in \mathscr{M}_{S}$.
    \end{enumerate}
  \end{lemma}
  \begin{proof}
        Assume that $\ov E\in \mathscr{M}_{F}$ and $\ocone(f)\in
      \mathscr{M}_{S}$. Let $g\colon \ov G\to \ocone(f)$ with $\ov G$
      and $\ocone(\ov G, \ocone(f))$ in $\mathscr{M}_{F}$. There is a
      natural isometry of complexes
      \begin{displaymath}
        \ocone(\ov G,\ocone(f)))\cong \ocone(\ocone(\ov G[-1],\ov
        E),\ov F)
      \end{displaymath}
      that shows $\ov F\in \mathscr{M}_{S}$.

      The second statement of the lemma is proved by a dual argument.
  \end{proof}

  Assume now that $f\colon \ov E\to \ov F$ is a morphism in $\oV(X)$
  and $\ov E,\ \ocone(f)\in
  \mathscr{M}_{S}$. Let $g\colon \ov G \to \ov E$ with $\ov G,\
  \ocone(g)\in \mathscr{M}_{F}$. There is a natural isometry
  \begin{displaymath}
    \ocone(\ocone(\ov G,\ov E),\ocone(\Id_{\ov F}))
    \cong
    \ocone(\ocone(\ov G,\ov F),\ocone(\ov E, \ov F)),
  \end{displaymath}
  that implies $\ocone(\ocone(\ov G,\ov F),\ocone(\ov E, \ov
  F))\in \mathscr{M}_{F}$. By Lemma \ref{lemm:11}, we deduce that
  $\ocone(\ov G,\ov F)\in \mathscr{M}_{S}$. By Lemma \ref{lemm:12},
  $\ov F\in  \mathscr{M}_{S}$.

  With $f$ as above, the fact that, if $\ov F$ and $\ocone(f)$ belong
  to $\mathscr{M}_{S}$ so does $\ov E$, is proved by a similar
  argument. In conclusion, $\mathscr{M}_{S}$ satisfies the condition
  \ref{def:9}~\ref{item:13}, hence $\mathscr{M}\subset
  \mathscr{M}_{S}$, which completes the proof of the
  theorem.
\end{proof}

The class of meager complexes satisfies the next list of properties,
that follow almost directly from Theorem \ref{thm:3}.
\begin{theorem} \label{thm:4}
  \begin{enumerate}
  \item \label{item:22} If $\overline E$ is a meager complex and $\ov
    F$ is a hermitian vector bundle, then the complexes
    $\ov F \otimes  \ov E$, $\Hom(\ov F,\ov E)$ and $\Hom(\ov E,\ov
    F)$, with the induced metrics, are meager.
  \item \label{item:23} If $\overline E^{\ast,\ast}$ is a bounded
    double complex of hermitian vector bundles and all rows (or
    columns) are meager complexes,
    then the complex $\Tot(\ov E^{\ast,\ast})$ is meager.
  \item \label{item:24} If $\ov E$ is a meager complex and $\ov F$ is another
    complex of hermitian vector bundles, then the complexes
    \begin{align*}
      \ov E\otimes \ov F&=\Tot((\ov F^{i}\otimes \ov E^{j})_{i,j}),\\
      \uHom (\ov E,\ov F)&= \Tot(\Hom ((\ov E^{-i},\ov F^{j})_{i,j}))\text{
        and }\\
      \uHom (\ov F,\ov E)&= \Tot(\Hom ((\ov F^{-i},\ov E^{j})_{i,j})),
    \end{align*}
    are meager.
  \item \label{item:25} If $f\colon X\to Y$ is a morphism of smooth complex
    varieties and $\ov E$ is a meager complex on $Y$, then $f^{\ast}
    \ov E$ is a meager complex on $X$.
  \end{enumerate}
\end{theorem}
We now introduce the notion of tight morphism.
\begin{definition}\label{def:tight_morphism}
A morphism $f\colon\ov E\to \ov F$ in $\oV(X)$ is said to be
\emph{tight} if
$\ocone(f)$ is a meager complex.
\end{definition}

\begin{proposition} \label{prop:8}
  \begin{enumerate}
  \item Every meager complex is acyclic.
  \item Every tight morphism is a quasi-isomorphism.
  \end{enumerate}
\end{proposition}
\begin{proof}
  Let $\overline E\in \mathscr{M}_{F}(X)$. Let $\Fil$ be any
  filtration that satisfies conditions
    \ref{def:12}~({\bf A}) and \ref{def:12}~({\bf B}).
   By definition, the
  complexes $\Gr_{\Fil}^{p}\overline E$ belong to $\mathscr{M}_{0}$,
  so they are acyclic. Hence
  $\overline E$ is acyclic.

  If $\overline E\in
  \mathscr{M}_{S}(X)$, let $\overline F$ and $\ocone(f)$
  be as in Definition \ref{def:12}~\ref{item:32}. Then, $\overline F$
  and
  $\ocone(f)$ are acyclic, hence $\overline E$ is also
  acyclic. Thus we have proved the first statement.
  The second statement is a direct consequence of the first one.
\end{proof}

Many arguments used for proving that a certain complex is meager or a certain
morphism is tight involve
cumbersome diagrams. In order to ease these arguments we will
develop a calculus of acyclic complexes.

Before starting we need some preliminary lemmas.

\begin{lemma} \label{lemm:16}
  Let $\ov E$, $\ov F$ be objects of $\oV(X)$. Then the following
  conditions are equivalent.
  \begin{enumerate}
  \item \label{item:37} There exists an object $\ov G$ and a diagram
    \begin{displaymath}
      \xymatrix{
        & \ov G \ar[dl]^{\sim}_{f} \ar[dr]^{g}&\\
        \ov E && \ov F,}
    \end{displaymath}
    such that $\ocone(g)\oplus \ocone(f)[1]$ is meager.
  \item \label{item:38}
    There exists an object $\ov G$ and a diagram
    \begin{displaymath}
      \xymatrix{
        & \ov G \ar[dl]^{\sim}_{f} \ar[dr]^{g}&\\
        \ov E && \ov F,}
    \end{displaymath}
    such that $f$ and $g$ are tight morphisms.
  \end{enumerate}
\end{lemma}
\begin{proof}
  Clearly, \ref{item:38} implies \ref{item:37}. To prove the converse
  implication, if $\ov G$ satisfies the conditions of \ref{item:37},
  we put $G'=G\oplus \ocone(f)$ and consider the morphisms $f'\colon \ov G'\to
  E$ and $g'\colon G'\to F$ induced by the first projection $G'\to
  G$. Then
  \begin{displaymath}
    \ocone(f')=\ocone(f)\oplus \ocone(f)[1],
  \end{displaymath}
  that is meager because $\ocone(f)$ is acyclic, and
  \begin{displaymath}
    \ocone(g')=\ocone(g)\oplus \ocone(f)[1],
  \end{displaymath}
  that is meager by hypothesis.
\end{proof}

\begin{lemma}\label{lemm:15}
Any diagram of tight morphisms, of the following types:
    \begin{equation}\label{diag:1}
      \begin{array}{ccc}
      \xymatrix{
        \ov E\ar[rd]_{f}	&	&\ov G\ar[ld]^{g}\\
        &\ov F
      } &
      \quad  \quad &
      \xymatrix{
		&\ov H\ar[rd]^{g'}\ar[ld]_{f'}	&\\
        \ov E	&	&\ov G
	}\\
    (i) && (ii)
    \end{array}
    \end{equation}
    can be completed into a diagram of tight morphisms
    \begin{equation}\label{eq:meager_1}
      \xymatrix{
        &\ov H\ar[ld]_{f'}\ar[rd]^{g'}	&\\
        \ov E	\ar[rd]_{f}&	&\ov G\ar[ld]^{g}\\
        &\ov F,	&
      }
    \end{equation}
which commutes up to homotopy.
\end{lemma}
\begin{proof}
  We prove the statement only for the case (i), the other one being analogous.
  Note that there is a natural arrow $\ov
  G\to\ocone(f)$. Define
  \begin{displaymath}
    \ov H=\ocone(\ov G,\ocone(f))[-1].
  \end{displaymath}
  With this choice, diagram (\ref{eq:meager_1} i) becomes commutative up
  to homotopy, taking the projection $H\to F[-1]$ as homotopy. We
  first show that $\ocone(\ov H, \ov G)$ is meager. Indeed, there is a
  natural isometry
  \begin{displaymath}
    \ocone(\ov H,\ov G)
    \cong\ocone(\ocone(\Id_{\ov G}), \ocone(\ov E,\ov F)[-1])
  \end{displaymath}
  and the right hand side complex is meager. Now for $\ocone(\ov H,
  \ov E)$. By Lemma \ref{lemm:13}, there is an isometry
  \begin{equation}
    \ocone(\ocone(\ov H, \ov E),\ocone(\ov G, \ov F))
    \cong
    \ocone(\ocone(\ov H, \ov G),\ocone(\ov E, \ov F)).
  \end{equation}
  The right hand side complex is meager, hence the left hand side is
  meager as well. Since, by hypothesis, $\ocone(\ov G, \ov F)$ is
  meager, the same is true for $\ocone(\ov H, \ov E)$.
\end{proof}

\begin{definition} \label{def:17}
  We will say that two complexes $\ov E$ and $\ov F$ are \emph{tightly
    related} if any of the equivalent conditions of Lemma
  \ref{lemm:16} holds.
\end{definition}
 It is easy to see, using Lemma \ref{lemm:15}, that to be tightly
related is an equivalence relation.

\begin{definition} \label{def:10}
  We denote by $\oV(X)/\mathscr{M}$ the set of classes of
  tightly related complexes. The class of a complex $\ov E$ will be
  denoted $[\ov E]$.
\end{definition}

\begin{theorem}[Acyclic calculus] \label{thm:7}
  \begin{enumerate}
  \item \label{item:34} For a complex $\ov E\in \Ob\oV(X)$, the class
    $[\ov E]=0$ if and only if $\ov E\in \mathscr{M}$.
  \item \label{item:33} The operation $\oplus$ induces an operation,
    that we denote $+$, in $\oV(X)/\mathscr{M}$. With this operation
    $\oV(X)/\mathscr{M}$ is an associative abelian semigroup.
  \item \label{item:39} For a complex $\ov E$, there exists a complex
    $\ov F$ such that $[\ov F]+[\ov E]=0$, if and only if $\ov E$ is acyclic. In
    this case $[\ov E[1]]=-[\ov E]$.
  \item \label{item:35} For every morphism $f\colon \ov E \to \ov F$,
    if $E$ is acyclic, then the equality
    \begin{displaymath}
      [\ocone(\ov E,\ov F)]=[\ov F]-[\ov E]
    \end{displaymath}
    holds.
  \item \label{item:40} For every morphism $f\colon \ov E \to \ov F$,
    if $F$ is acyclic, then the equality
    \begin{displaymath}
      [\ocone(\ov E,\ov F)]=[\ov F]+[\ov E[1]]
    \end{displaymath}
    holds.
  \item \label{item:36} Given a diagram
    \begin{displaymath}
      \xymatrix{
        \overline{E}'\ar[r]^{f'}\ar[d]_{g'} &\overline{F}'\ar[d]^{g}\\
        \overline{E}\ar[r]^{f}	&\overline{F}
      }
    \end{displaymath}
    in $\oV(X)$, that commutes up to homotopy, then for every choice
    of homotopy we have
    \begin{displaymath}
      [\ocone(\ocone(f'),\ocone(f))]=[\ocone(\ocone(-g'),\ocone(g))].
    \end{displaymath}
  \item \label{item:41} Let $f\colon \ov E\to \ov F$, $g\colon \ov
    F\to \ov G$ be morphisms of complexes. Then
    \begin{align*}
      [\ocone(\ocone(g\circ f),\ocone(g))]&=[\ocone(f)[1]],\\
      [\ocone(\ocone(f),\ocone(g\circ f))]&=[\ocone(g)].
    \end{align*}
    If one of $f$ or $g$ are quasi-isomorphisms, then
    \begin{displaymath}
      [\ocone(g\circ f)]=[\ocone(g)]+[\ocone(f)].
    \end{displaymath}
    If $g\circ f$ is a quasi-isomorphism, then
    \begin{displaymath}
      [\ocone(g)]=[\ocone(f)[1]]+[\ocone(g\circ f)].
    \end{displaymath}
  \end{enumerate}
\end{theorem}
\begin{proof}
  The statements \ref{item:34} and \ref{item:33} are immediate.
  For assertion \ref{item:39}, observe that, if $\ov E$ is
  acyclic, then $\ov E\oplus \ov E[1]$ is meager. Thus
  \begin{displaymath}
    [\ov E]+[\ov E[1]]=[\ov E\oplus \ov E[1]]=0.
  \end{displaymath}
  Conversely, if $[\ov F]+[\ov E]=0$, then $\ov F\oplus \ov E$ is
  meager, hence acyclic. Thus $\ov E$ is acyclic.

  For property \ref{item:35} we consider the map $\ov F\oplus \ov E[1]\to
  \ocone(f)$ defined by the map $\ov F\to \ocone(f)$. There is a
  natural isometry
  \begin{displaymath}
    \ocone(\ov F\oplus \ov E[1],\ocone(f))\cong
    \ocone(\ov E\oplus \ov E[1],\ocone(\Id_{F})).
  \end{displaymath}
  Since the right hand complex is meager, so is the first. In consequence
  \begin{displaymath}
    [\ocone(f)]=[\ov F\oplus \ov E[1]]=[\ov F]+[\ov E[1]]=[\ov F]-[\ov E].
  \end{displaymath}
  Statement \ref{item:40} is proved analogously.

  Statement \ref{item:36} is a direct consequence of Lemma
  \ref{lemm:13}.

  Statement \ref{item:41} is an easy consequence of the previous
  properties.
\end{proof}

\begin{remark}\label{rem:9}
  In $f\colon \ov E \to \ov F$ is a morphism and neither $\ov E$ nor
  $\ov F$ are acyclic, then $[\ocone(f)]$ depends on the homotopy
  class of $f$ and not only on $\ov E$ and $\ov F$. For instance, let
  $\ov E$ be a non-acyclic complex of hermitian bundles. Consider the
  zero map and the identity map $0,\Id\colon \ov E\to \ov E$.  Since,
  by Example \ref{exm:2}, we know that $\ocone (\Id)$ is
  meager, then $[\ocone (\Id)]=0$. By contrast,
  \begin{displaymath}
    [\ocone(0)]=[\ov E]+[\ov E[-1]]\not = 0
  \end{displaymath}
  because $\ov E$ is not acyclic. This implies that we can not extend
  Theorem \ref{thm:7}~\ref{item:35} or \ref{item:40} to the case
  when none of the complexes are acyclic.
\end{remark}

\begin{corollary}\label{cor:3}
  \begin{enumerate}
  \item \label{item:42}
    Let
    \begin{displaymath}
      0\longrightarrow \ov E\longrightarrow \ov F\longrightarrow \ov G
      \longrightarrow 0
    \end{displaymath}
    be a short exact sequence in $\oV(X)$ all whose constituent short
    exact sequences are orthogonally split.  If either $\ov E$ or
    $\ov G$ is acyclic, then
    \begin{displaymath}
      [\ov F]=[\ov E]+[\ov G].
    \end{displaymath}
  \item \label{item:58} Let $\ov E^{\ast,\ast}$ be a bounded double complex of
    hermitian vector bundles. If the columns of $\ov E^{\ast,\ast}$
    are acyclic, then
    \begin{displaymath}
      [\Tot(\ov E^{\ast,\ast})]=\sum_{k}(-1)^{k} [\ov E^{k,\ast}].
    \end{displaymath}
    If the rows are acyclic, then
    \begin{displaymath}
      [\Tot(\ov E^{\ast,\ast})]=\sum_{k}(-1)^{k} [\ov E^{\ast,k}].
    \end{displaymath}
    In particular, if rows and columns are acyclic
    \begin{displaymath}
      \sum_{k}(-1)^{k} [\ov E^{k,\ast}]=\sum_{k}(-1)^{k} [\ov E^{\ast,k}].
    \end{displaymath}
  \end{enumerate}
\end{corollary}
\begin{proof}
  The first item follows from Theorem \ref{thm:7}~\ref{item:35}
  and \ref{item:40}, by using Remark \ref{rem:2}. The second assertion
  follows from the first by induction on the size of the complex, by
  using the usual filtration of $\Tot(E^{\ast,\ast})$.
\end{proof}

As an example of the use of the acyclic calculus we prove
\begin{proposition}\label{prop:6}
  Let $f\colon \ov E\to\ov F$ and $g\colon \ov F\to\ov G$ be morphisms
  of complexes. If two of $f,\ g,\ g\circ f$ are tight, then so is
  the third.
\end{proposition}
\begin{proof}
  Since tight morphisms are quasi-isomorphisms, by Theorem
  \ref{thm:7}~\ref{item:41}
  \begin{displaymath}
    [\ocone(g\circ f)]=[\ocone(f)]+[\ocone(g)].
  \end{displaymath}
  Hence the result follows from \ref{thm:7}~\ref{item:34}.
\end{proof}

\begin{definition}\label{def:KA}
  We will denote by $\KA(X)$ the set of invertible elements of
  $\oV(X)/\mathscr{M}$. This is an abelian subgroup. By Theorem
  \ref{thm:7}~\ref{item:39} the group $\KA(X)$
  agrees with the image in $\oV(X)/\mathscr{M}$ of the class of
  acyclic complexes.
\end{definition}

The group $\KA(X)$ is a universal abelian group for additive Bott-Chern
classes. More precisely, let us denote by $\oVo(X)$ the
full subcategory of $\oV(X)$ of acyclic complexes.

\begin{theorem} \label{thm:8}
  Let $\mathscr{G}$ be an abelian group and let $\varphi\colon \Ob \oVo(X) \to
  \mathscr{G}$ be an assignment such that
  \begin{enumerate}
  \item (Normalization) Every complex of the form
    \begin{displaymath}
      \ov E\colon\quad 0\longrightarrow \ov A
      \overset{\Id}{\longrightarrow}
      \ov A \longrightarrow 0
    \end{displaymath}
satisfies $\varphi(\ov E)=0$.
  \item (Additivity for exact sequences) For every short exact
    sequence in $\oVo(X)$
    \begin{displaymath}
      0\longrightarrow \ov E\longrightarrow \ov F\longrightarrow \ov G
      \longrightarrow 0,
    \end{displaymath}
    all whose constituent short
    exact sequences are orthogonally split, we have
    \begin{displaymath}
      \varphi(\ov F)=\varphi(\ov E)+\varphi(\ov G).
    \end{displaymath}
  \end{enumerate}
Then $\varphi$ factorizes through a group homomorphism $\widetilde
\varphi\colon \KA(X)\to \mathscr{G}$.
\end{theorem}
\begin{proof}
  The second condition tells us that $\varphi$ is a morphism of
  semigroups. Thus we only need to show that it vanishes on meager
  complexes.

  Again by the second condition, it is enough to prove that $\varphi$
  vanishes on the class $\mathscr{M}_{0}$. Both conditions together
  imply that $\varphi$ vanishes on orthogonally split
  complexes. Therefore, by Example
  \ref{exm:2}, it vanishes on complexes of the form
  $\ocone(\Id_{E})$. Once more by the second condition, if $E$ is acyclic,
  \begin{displaymath}
    \varphi(E)+\varphi(E[1])=\varphi(\ocone(\Id_{E}))=0.
  \end{displaymath}
  Thus $\varphi$ vanishes also on the complexes described in Definition \ref{def:11} (ii).
  Hence $\varphi$ vanishes on the class $\mathscr{M}$.
\end{proof}

\begin{remark}
  The considerations of this section carry over to the category of
  complex analytic varieties. If $M$ is a complex analytic variety, one thus obtains for
  instance a group $\KA^{\an}(M)$. Observe that, by GAGA principle, whenever $X$ is a
  proper smooth algebraic variety over $\CC$, the group $\KA^{\an}(X^{\an})$ is
  canonically isomorphic to $\KA(X)$.
\end{remark}

As an example, we consider the simplest case $\Spec \CC$ and we
compute the group $\KA(\Spec \CC)$. Given an acyclic
complex $E$ of $\CC$-vector spaces, there is a canonical isomorphism
\begin{displaymath}
  \alpha :\det E\longrightarrow \CC.
\end{displaymath}
If we have an acyclic complex of hermitian vector bundles $\ov E$,
there is an induced metric on $\det E$. If we put on $\CC$ the trivial hermitian metric, then there is a well
defined positive real number $\|\alpha \|$, namely the norm of the isomorphism $\alpha$.

\begin{theorem}
  The assignment $\ov E\mapsto \log\|\alpha \|$ induces an isomorphism
  \begin{displaymath}
    \widetilde{\tau }\colon \KA(\Spec
    \CC)\overset{\simeq}{\longrightarrow}\RR.
  \end{displaymath}
\end{theorem}
\begin{proof}
  First, we observe that the assignment in the theorem
  satisfies the hypothesis of Theorem \ref{thm:8}. Thus,
  $\widetilde{\tau }$ exists and is a group morphism.
  Second, for every $a\in \RR$ we consider the acyclic complex
  \begin{displaymath}
    e^{a}:= 0\longrightarrow
    \overline{\CC}\overset{e^{a}}{\longrightarrow}
    \overline{\CC}\longrightarrow 0,
  \end{displaymath}
  where $\overline {\CC}$ has the standard metric and the left copy of
  $\ov {\CC}$ sits in degree $0$.
  Since $\widetilde {\tau }([e^{a}])=a$ we deduce that $\widetilde{\tau }$ is
  surjective.

  Next we prove that the complexes of the form $[e^{a}]$ form a set of
  generators of $\KA(\Spec \CC)$.
  Let
  $\ov E=(\ov E^{\ast},f^{\ast})$ be an acyclic complex. Let $r=\sum_{i}
  \rk(E^{i})$. We will show by induction on $r$ that
  $[\ov E]=\sum_{k}(-1)^{i_{k}}[e^{a_{k}}]$ for certain integers
  $i_{k}$ and real numbers $a_{k}$. Let $n$ be the
  smallest integer such that $f^{n}\colon E^{n}\to E^{n+1}$ is
  non-zero. Let $v\in E^{n}\setminus \{0\}$. By acyclicity, $f^{n}$ is
  injective, hence $\|f^{n}(v)\|\not=0.$ Set $i_{1}=n$ and
  $a_{1}=\log(\|f^{n}(v)\|/\|v\|)$ and consider the diagram
  \begin{displaymath}
    \xymatrix{
      & &0\ar[d]  &0\ar[d] &  &\\
      &0\ar[r] &\ov{\CC}\ar[r]^{e^{a}}\ar[d]^{\gamma ^{n}}
      &\ov{\CC}\ar[r]\ar[d]^{\gamma ^{n+1}}
      &0 \ar[r]\ar[d]
      & \dots
      \\
      &0\ar[r] &\ov E^{n}\ar[r]\ar[d] &\ov E^{n+1}\ar[r]\ar[d] &\ov E^{n+2}\ar[r]\ar[d] &\dots
      \\
      &0\ar[r] &\ov F^{n}\ar[r]\ar[d] &\ov F^{n+1}\ar[r]\ar[d] &\ov F^{n+2}\ar[r]\ar[d] &\dots
      \\
      & &0 &0 &0 & &
    }
  \end{displaymath}
  where $\gamma ^{n}(1)=v$, $\gamma ^{n+1}(1)=f^{n}(v)$ and all the
  columns are orthogonally split short exact sequences. By Corollary
  \ref{cor:3}~\ref{item:42} and Theorem \ref{thm:7}~\ref{item:39}, we have
  \begin{displaymath}
    [\ov E]=(-1)^{i_{1}}[e^{a_{1}}]+[\ov F].
  \end{displaymath}
  Thus we deduce the claim.

  Considering now the diagram
  \begin{displaymath}
    \xymatrix{ \ov \CC \ar[r]^{e^{a}}\ar[d]_{\Id}& \ov \CC\ar[d]^{e^{b}}\\
       \ov \CC \ar[r]_{e^{a+b}}& \ov \CC
    }
  \end{displaymath}
  and using Corollary \ref{cor:3}~\ref{item:58} we deduce that
  $[e^{a}]+[e^{b}]=[e^{a+b}]$ and  $[e^{-a}]=-[e^{a}]$. Therefore
  every element of $\KA(\Spec \CC)$ is of the form $[e^{a}]$. Hence
  $\widetilde{\tau }$ is also injective.
\end{proof}

\section{Definition of $\oDb(X)$ and basic
  constructions}\label{sec:oDb}
Let $X$ be a smooth algebraic variety over $\CC$. We denote by
$\Coh(X)$ the abelian category of coherent sheaves on $X$ and by
$\Db(X)$ its bounded derived category.
The objects of $\Db(X)$ are complexes of quasi-coherent sheaves
with bounded coherent cohomology.
The reader is referred to \cite{GelfandManin:MR1950475} for an introduction to derived categories.
For notational convenience, we also introduce $\Cb(X)$,
the abelian category of bounded cochain complexes of coherent sheaves
on $X$. Arrows in $\Db(X)$ will be written as $\dashrightarrow$, while
arrows in $\Cb(X)$ will be denoted by $\rightarrow$. The symbol $\sim$
will mean either quasi-isomorphism in $\Cb(X)$ or isomorphism in
$\Db(X)$. Every functor from $\Db(X)$ to another category will tacitly
be assumed to be the derived functor. Therefore we will denote just by
$\Rd f_{\ast}$, $\Ld f^{\ast}$, $\Lotimes$ and $\RuHom$ the derived
direct image, inverse image, tensor product and internal Hom.
Finally, we will refer to (complexes of) locally free sheaves by normal upper case
letters (such as $F$) whereas we reserve script upper case letters for
(complexes of) quasi-coherent sheaves in general (for instance $\mathcal{F}$).

\begin{remark}\label{rem:3}
  Because $X$ is in particular a smooth noetherian scheme over $\CC$,
  every object $\mathcal{F}$ of $\Cb(X)$ admits a quasi-isomorphism
  $F\to \mathcal{F}$, with $F$ an object of $\Vb(X)$. Hence, if
  $\mathcal{F}$ is an object in $\Db(X)$, then there is an
  isomorphism $F\dra\mathcal{F}'$ in $\Db(X)$, for some
  object  $F\in\Vb(X)$. In general, the analogous
  statement is no longer true if we work with complex manifolds, as
  shown by the counterexample \cite[Appendix,
  Cor. A.5]{Voisin:counterHcK}.
\end{remark}

For the sake of completeness, we recall how morphisms in $\Db(X)$
between bounded complexes of vector bundles can be represented.
\begin{lemma}\label{lemma:morphisms_Db}
  \begin{enumerate}
  \item \label{item:26} Let $F, G$ be bounded complexes of vector
    bundles on $X$. Every morphism $F\dra G$ in $\Db(X)$ may be
    represented by a diagram in $\Cb(X)$
    \begin{displaymath}
      \xymatrix{
        &E\ar[ld]_{f}\ar[rd]^{g}	&\\
        F	&	&G,
      }
    \end{displaymath}
    where $E\in \Ob \Vb(X)$ and $f$ is a quasi-iso\-morphism.
  \item \label{item:27} Let $E$, $E'$, $F$, $G$ be bounded
    complexes of 
    vector bundles on $X$. Let $f$, $f'$, $g$, $g'$
    be morphisms in $\Cb(X)$ as in the diagram below, with $f$, $f'$
    quasi-isomorphisms.  These data define the same morphism $F \dra
    G$ in $\Db(X)$ if, and only if, there exists a bounded complex of
    vector bundles $E''$ and a diagram
    
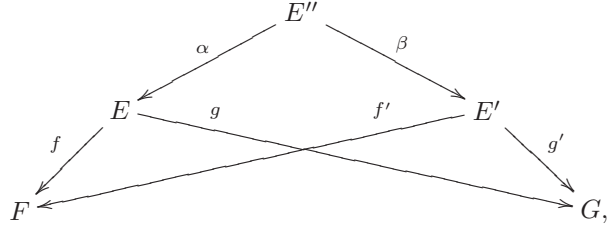
\begin{displaymath}
      \xymatrix{
	& & E''\ar[ld]_{\alpha}\ar[rd]^{\beta} &	&\\
	&E\ar[ld]_{f}\ar[rrrd]	&{}_g\hspace{2cm} {}_{f'}
        &E'\ar[llld]\ar[rd]^{g'} &\\
	F & &	& &G,
      }
    \end{displaymath}
    whose squares are commutative up to homotopy and where $\alpha$
    and $\beta$ are quasi-isomorphisms.
  \end{enumerate}
\end{lemma}
\begin{proof}
  This follows from the equivalence of $\Db(X)$ with the localization
  of the homotopy category of $\Cb(X)$ with respect to the class of
  quasi-isomorphisms and Remark \ref{rem:3}.
\end{proof}

\begin{proposition} \label{prop:7} Let $f\colon \ov E\to \ov E$ be an
  endomorphism in $\oV(X)$ that represents $\Id_{E}$ in $\Db(X)$. Then
  $\ocone(f)$ is meager.
\end{proposition}
\begin{proof}
  By Lemma \ref{lemma:morphisms_Db}~\ref{item:27}, the fact that $f$
  represents the identity in $\Db(X)$ means that there are diagrams
  \begin{displaymath}
    \xymatrix{
      \ov E' \ar[r]^{\alpha }_{\sim} \ar[d]_{\beta }^{\sim}& \ov E \ar[d]^{\Id_{E}} &&
      \ov E' \ar[r]^{\alpha }_{\sim} \ar[d]_{\beta }^{\sim}& \ov E\ar[d]^{f}\\
      \ov E \ar[r]_{\Id_{E}}& \ov E, && \ov E \ar[r]_{\Id_{E}}& \ov E,
    }
  \end{displaymath}
  that commute up to homotopy. By Theorem \ref{thm:7}~\ref{item:35}
  and \ref{item:36} the equalities
  \begin{align*}
    [\ocone(\alpha )]-[\ocone(\Id_{E})]&=
    [\ocone(\beta )]-[\ocone(\Id_{E})]\\
    [\ocone(\alpha )]-[\ocone(\Id_{E})]&= [\ocone(\beta )]-[\ocone(f)]
  \end{align*}
  hold in the group $\KA(X)$ (observe that these relations do not
  depend on the choice of homotopies).Therefore
  \begin{displaymath}
    [\ocone(f)]=[\ocone(\Id_{E})]=0.
  \end{displaymath}
  Hence $\ocone(f)$ is meager.
\end{proof}

\begin{definition}\label{definition:hermitian_structure}
  Let $\mathcal{F}$ be an object of $\Db(X)$. A \textit{hermitian
    metric on} $\mathcal{F}$ consists of the following data:
  \begin{enumerate}
  \item[--] an isomorphism
    $E\overset{\sim}{\dashrightarrow}\mathcal{F}$ in $\Db(X)$, where
    $E\in \Ob \Vb(X)$;
  \item[--] an object $\ov E\in \Ob \oV(X)$, whose underlying complex
    is $E$.
  \end{enumerate}
  We write $\overline{E}\dashrightarrow\mathcal{F}$ to refer to the
  data above and we call it a \textit{metrized object of} $\Db(X)$.
\end{definition}

Our next task is to define the category $\oDb(X)$, whose objects are
objects of $\Db(X)$ provided with equivalence classes of metrics. We
will show that in this category there is a hermitian cone well defined
up to isometries.

\begin{lemma}\label{lemm:14}
  Let $\ov E, \ov E'\in \Ob(\oV(X))$ and consider an arrow $E\dra E'$ in
  $\Db(X)$.
  Then the following statements are equivalent:
  \begin{enumerate}
  \item \label{item:30} for any diagram
    \begin{equation}\label{diag:2}
      \xymatrix{
        & E'' \ar[dl]_{\sim} \ar[dr]&\\
        E && E',}
    \end{equation}
    that represents $E \dra E'$, and any choice of hermitian metric on
    $E''$, the complex
    \begin{equation}\label{eq:54}
      \ocone(\ov E'',\ov E)[1]\oplus \ocone(\ov E'',\ov E')
    \end{equation}
    is meager;
  \item \label{item:31} there is a diagram (\ref{diag:2})
    that represents $E \dra E'$, and a choice of hermitian metric on
    $E''$, such that the complex (\ref{eq:54}) is meager;
  \item \label{item:31bis}there is a diagram (\ref{diag:2})
    that represents $E \dra E'$, and a choice of hermitian metric on
    $E''$, such that the arrows $\ov E''\to \ov E$ and $\ov E'\to \ov
    E'$ are tight morphisms.
  \end{enumerate}
\end{lemma}
\begin{proof}
  Clearly \ref{item:30} implies \ref{item:31}. To prove the converse
  we assume the existence of a $\ov E''$ such that the complex
  \eqref{eq:54} is meager, and let $\ov E'''$ be any complex that satisfies the
  hypothesis of \ref{item:30}. Then there is a diagram
  \begin{displaymath}
    \xymatrix{
      & & E^{''''}\ar[ld]_{\alpha}\ar[rd]^{\beta} &	&\\
      &E''\ar[ld]_{f}\ar[rrrd]	&{}_g\hspace{2cm} {}_{f'}
      &E'''\ar[llld]\ar[rd]^{g'} &\\
      E & &	& &E'
    }
  \end{displaymath}
    whose squares commute up to homotopy. Using acyclic calculus we have
  \begin{multline*}
    [\ocone(g')]-[\ocone(f')]=\\
    [\ocone(\beta )]+[\ocone(g)]-[\ocone(\alpha )]
    -[\ocone(\beta )]-[\ocone(f)]+[\ocone(\alpha )]=\\
    [\ocone(g)]-[\ocone(f)]=0.
  \end{multline*}
  Now repeat the argument of Lemma \ref{lemm:16} to prove that
  \ref{item:31} and \ref{item:31bis} are equivalent. The only point is
  to observe that the diagram constructed in Lemma \ref{lemm:16}
  represents the same morphism in the derived category as the
  original diagram.
\end{proof}

\begin{definition}\label{def:13}
  Let $\mathcal{F}\in \Ob\Db(X)$ and let $\ov E\dra \mathcal{F}$ and
  $\ov E'\dra \mathcal{F}$ be two hermitian metrics on
  $\mathcal{F}$. We say that they \emph{fit tight} if the induced
  arrow $\ov E\dra\ov E'$ satisfies any of the equivalent conditions
  of Lemma \ref{lemm:14}
\end{definition}

\begin{theorem} \label{thm:5} The relation ``to fit tight'' is an
  equivalence relation.
\end{theorem}
\begin{proof}

  The reflexivity and the symmetry are obvious. To prove the
  transitivity, consider a diagram
  \begin{displaymath}
    \xymatrix{
      & \ov F \ar[dl]_{f} \ar[dr]^{g}& & \ov F' \ar[dl]_{f'} \ar[dr]^{g'} &\\
      \ov E & & \ov E' & & \ov E'' ,
    }
  \end{displaymath}
  where all the arrows are tight morphisms and $f$, $f'$ are quasi-isomorphisms. By Lemma \ref{lemm:15},
  this diagram can be completed into a diagram
  \begin{displaymath}
    \xymatrix{
      & & \ov F'' \ar[dl]_{\alpha } \ar[dr]^{\beta }& & \\
      & \ov F \ar[dl]_{f} \ar[dr]^{g}& & \ov F' \ar[dl]_{f'} \ar[dr]^{g'} &\\
      \ov E & & \ov E' & & \ov E'' ,
    }
  \end{displaymath}
  where all the arrows are tight morphisms and the square commutes up
  to homotopy. Now observe that $f\circ\alpha$ and $g'\circ\beta$
  represent the morphism $E\dra E''$ in $\Db(X)$ and are both tight
  morphisms by Proposition \ref{prop:6}. This finishes the proof.

\end{proof}

\begin{definition}\label{def:category_oDb}
  We denote by $\oDb(X)$ the category whose objects are pairs
  $\ocF=(\mathcal{F},h)$ where $\mathcal{F}$ is an object of $\Db(X)$
  and $h$ is an equivalence class of metrics that fit tight, and with
  morphisms
  \begin{displaymath}
    \Hom_{\oDb(X)}(\overline{\mathcal{F}},\overline{\mathcal{G}})
    =
    \Hom_{\Db(X)}(\mathcal{F},\mathcal{G}).
  \end{displaymath}
  A class $h$ of metrics will be called \emph{a hermitian structure},
  and may be referenced by any representative $\ov E\dra \mathcal{F}$
  or, if the arrow is clear, by the complex $\ov E$. We will denote by
  $\ov 0\in \Ob \oDb(X)$ a zero object of $\Db(X)$ provided with a
  trivial hermitian structure given by any meager complex.

  If the underlying complex to an object $\ov{\mathcal{F}}$ is
  acyclic, then its hermitian structure has a well defined class in
  $\KA(X)$. We will use the notation $[\ov{\mathcal{F}}]$ for this
  class.
\end{definition}
\begin{definition}
  A morphism in $\oDb(X)$, $f\colon (\ov E\dashrightarrow\mathcal{F})
  \dashrightarrow(\ov F\dra \mathcal{G})$, is called \emph{a tight
    isomorphism} whenever the underlying morphism $f\colon
  \mathcal{F}\dra \mathcal{G}$ is an isomorphism and the metric on
  $\mathcal{G}$ induced by $f$ and $\ov E$ fits tight with $\ov F$. An
  object of $\oDb(X)$ will be called \emph{meager} if it is tightly
  isomorphic to the zero object with the trivial metric.
\end{definition}
\begin{remark} \label{rem:5} A word of warning should be said about
  the use of acyclic calculus to show that a particular map is a tight
  isomorphism. There is an assignment $\Ob\oDb(X)\to
  \oV(X)/\mathscr{M}$ that sends $\ov E\dra \mathcal{F}$ to $[\ov
  E]$. This assignment is not injective. For instance, let $r>0$ be a
  real number and consider the trivial bundle $\mathcal{O}_{X}$ with
  the trivial metric $\|1\|=1$ and with the metric $\|1\|'=1/r$. Then
  the product by $r$ induces an isometry between both bundles. Hence,
  if $\ov E$ and $\ov E'$ are the complexes that have
  $\mathcal{O}_{X}$ in degree $0$ with the above hermitian metrics,
  then $[\ov E]=[\ov E']$, but they define different hermitian
  structures on $\mathcal{O}_{X}$ because the product by $r$ does not
  represent $\Id_{\mathcal{O}_{X}}$.

  Thus the right procedure to show that a morphism $f\colon (\ov E\dra
  \mathcal{F})\dra (\ov F\dra \mathcal{G})$ is a tight isomorphism, is
  to construct a diagram
  \begin{displaymath}
    \xymatrix{
      & \ov G \ar[dl]^{\sim}_{\alpha } \ar[dr]^{\beta }&\\
      \ov E && \ov F}
  \end{displaymath}
  that represents $f$ and use the acyclic calculus to show that
  $[\ocone(\beta )]-[\ocone(\alpha )]=0$.
\end{remark}

By definition, the forgetful functor $\mathfrak{F}\colon \oDb(X)\to
\Db(X)$ is fully faithful. The structure of this functor will be given
in the next result that we suggestively summarize by saying that
$\oDb(X)$ is a principal fibered category over $\Db(X)$ with
structural group $\KA(X)$ provided with a flat connection.

\begin{theorem}\label{thm:13}
  The functor $\mathfrak{F}\colon \oDb(X)\to \Db(X)$ defines a
  structure of category fibered in grupoids. Moreover
  \begin{enumerate}
  \item \label{item:43} The fiber $\mathfrak{F}^{-1}(0)$ is the
    grupoid associated to the abelian group
    $\KA(X)$. The object $\ov 0$ is the neutral element of $\KA(X)$.
  \item \label{item:44} For any object $\mathcal{F}$ of $\Db(X)$, the
    fiber $\mathfrak{F}^{-1}(\mathcal{F})$ is the grupoid associated
    to a torsor over $\KA(X)$. The action of $\KA(X)$ over
    $\mathfrak{F}^{-1}(\mathcal{F})$ is given by orthogonal direct
    sum. We will denote this action by $+$.
  \item \label{item:45} Every isomorphism $f\colon \mathcal{F}\dra
    \mathcal{G}$ in $\Db(X)$ determines an isomorphism of
    $\KA(X)$-torsors
    \begin{displaymath}
      \mathfrak{t}_{f}\colon \mathfrak{F}^{-1}(\mathcal{F})\longrightarrow
      \mathfrak{F}^{-1}(\mathcal{G}),
    \end{displaymath}
    that sends the hermitian structure $\ov E\overset{\epsilon }{\dra}
    \mathcal{F}$ to the hermitian structure $\ov E\overset{f\circ
      \epsilon }{\dra} \mathcal{G}$. This isomorphism will be called
    the parallel transport along $f$.
  \item \label{item:46} Given two isomorphisms $f\colon
    \mathcal{F}\dra \mathcal{G}$ and $g\colon \mathcal{G}\dra
    \mathcal{H}$, the equality
  $$\mathfrak{t}_{g\circ f}=\mathfrak{t}_{g}\circ \mathfrak{t}_{f}$$
  holds.
\end{enumerate}

\end{theorem}
\begin{proof} Recall that $\mathfrak{F}^{-1}(\mathcal{F})$ is the
  subcategory of $\oDb(X)$ whose objects satisfy
  $\mathfrak{F}(A)=\mathcal{F}$ and whose morphisms satisfy
  $\mathfrak{F}(f)=\Id_{\mathcal{F}}$. The first assertion is
  trivial. To prove that $\mathfrak{F}^{-1}(\mathcal{F})$ is a torsor
  under $\KA(X)$, we need to show that $\KA(X)$ acts freely and
  transitively on this fiber. For the freeness, it is enough to
  observe that if for $\ov E\in\oV(X)$ and $\ov M\in\oVo(X)$, the
  complexes $\ov E$ and $\ov E\oplus \ov M$ represent the same
  hermitian structure, then the inclusion $\ov E\hookrightarrow \ov
  E\oplus\ov M$ is tight. Hence $\ocone(\ov E, \ov E\oplus \ov M)$ is
  meager. Since
  \begin{displaymath}
    \ocone(\ov E, \ov E\oplus \ov M)=\ocone(\ov E, \ov E)\oplus \ov M
  \end{displaymath}
  and $\ocone(\ov E, \ov E)$ is meager, we deduce that $\ov M$ is
  meager. For the transitivity, any two hermitian structures on
  $\mathcal{F}$ are related by a diagram
  \begin{displaymath}
    \xymatrix{
      &\ov E''\ar[ld]_{\sim}^{f}\ar[rd]^{\sim}_{g}	&\\
      \ov E	& &\ov E'.
    }
  \end{displaymath}
  After possibly replacing $\ov E''$ by $\ov E''\oplus\ocone(f)$, we
  may assume that $f$ is tight. We consider the natural arrow $\ov
  E''\to \ov E'\oplus\ocone(g)[1]$ induced by $g$. Observe that
  $\ocone(g)[1]$ is acyclic. Finally, we find
  \begin{displaymath}
    \ocone(\ov E'', \ov E'\oplus\ocone(g)[1])=\ocone(g)\oplus\ocone(g)[1],
  \end{displaymath}
  that is meager. Thus the hermitian structure represented by $\ov
  E''$ agrees with the hermitian structure represented by $\ov
  E'\oplus\ocone(g)[1]$.

  The remaining properties are straightforward.
\end{proof}

Our next objective is to define the cone of a morphism in
$\oDb(X)$. This will be an object of $\oDb(X)$ uniquely defined up to
tight isomorphism. Let $f\colon
(\overline{E}\dashrightarrow\mathcal{F})\dra
(\overline{E}'\dashrightarrow\mathcal{G})$
be a morphism in $\oDb(X)$, where $\ov E$ and $\ov E'$ are
representatives of the hermitian structures.

 \begin{definition}\label{def:her_cone}
   A \textit{hermitian cone} of $f$, denoted $\ocone(f)$, is an
   object $(\cone(f),h_{f})$ of $\oDb(X)$ where:
   \begin{enumerate}
   \item[--] $\cone(f)\in\Ob\Db(X)$ is a choice of cone of $f$. Namely
     an object of $\Db(X)$ completing $f$ into a distinguished
     triangle;
   \item[--] $h_{f}$ is a hermitian structure on $\cone(f)$
     constructed as follows. The morphism $f$ induces an arrow $E\dra
     E'$. Choose any bounded complex $E''$ of vector bundles with a
     diagram
     \begin{displaymath}
       \xymatrix{
         &E''\ar[ld]^{\sim}\ar[rd]	&\\
         E	&	&E'
       }
     \end{displaymath}
     that represents $E\dra E'$, and an arbitrary hermitian metric on
     $E''$. Put
     \begin{equation}\label{eq:her_cone}
       \ov{C}(f)=\ocone(\ov{E}'',\ov{E})[1]\oplus\ocone(\ov{E}'',\ov{E}').
     \end{equation}
     There are morphisms defined as compositions
     \begin{displaymath}
       \ov{E}'\longrightarrow \ocone(\ov{E}'',\ov{E}')
       \longrightarrow \ov{C}(f),
     \end{displaymath}
     where the second arrow is the natural inclusion, and
     \begin{displaymath}
       \ov{C}(f) \longrightarrow \ocone(\ov{E}'',\ov{E}')
       \longrightarrow \ov{E}''[1] \longrightarrow
       \ov{E}[1],
     \end{displaymath}
     where the first arrow is the natural projection.
     These morphisms fit into a natural distinguished triangle completing
     $\ov{E}\dra\ov{E}'$. By the axioms of triangulated category,
     there is a quasi-isomorphism $\ov{C}(f)\dra\cone(f)$ such that
     the following diagram (where the rows are distinguished triangles)
     \begin{displaymath}
       \xymatrix{
         \ov{E}\ar@{-->}[r]\ar@{-->}[d]	&\ov{E}'\ar@{-->}[r]\ar@{-->}[d]
         &\ov{C}(f)\ar@{-->}[d]\ar@{-->}[r]	&\ov{E}[1]\ar@{-->}[d]\\
         \mathcal{F}\ar@{-->}[r]	&\mathcal{G}\ar@{-->}[r]
         &\cone(f)\ar@{-->}[r]	&\mathcal{F}[1]
       }
     \end{displaymath}
     commutes. We take the hermitian structure that
     $\ov{C}(f)\dra\cone(f)$ defines on $\cone(f)$. By Theorem
     \ref{thm:6bis} below, this hermitian structure does not depend on
     the particular choice of arrow $\ov{C}(f)\dra\cone(f)$. Moreover,
     by Theorem \ref{thm:6}, the hermitian structure will not depend
     on the choices of representatives of hermitian structures nor on
     the choice of $\ov{E}''$.
   \end{enumerate}
 \end{definition}
 \begin{remark}\label{rem:8}
   The factor $\ocone(\ov{E}'',\ov{E})[1]$ has to be seen as a
   correction term to take into account the difference of metrics from
   $\ov E $ and $\ov E''$. We would have obtained an equivalent
   definition using the factor $\ocone(\ov{E}'',\ov{E})[-1]$.
 \end{remark}

 \begin{theorem}\label{thm:6bis}
   Let
   \begin{displaymath}
     \xymatrix{
       \mathcal{F}\ar@{-->}[r]\ar[d]^{\Id}	
       &\mathcal{G}\ar@{-->}[r]\ar[d]^{\Id}
       &\mathcal{H}\ar@{-->}[r]\ar@{-->}[d]^{\alpha }
       & \mathcal{F}[1]\ar@{-->}[r]	\ar[d]^{\Id}
       &\dots\\
       \mathcal{F}\ar@{-->}[r]	
       &\mathcal{G}\ar@{-->}[r]
       &\mathcal{H}\ar@{-->}[r]
       & \mathcal{F}[1]\ar@{-->}[r]
       &\dots
     }
   \end{displaymath}
   be a commutative diagram in $\Db(X)$, where the rows are the same
   distinguished triangle. Let $\ov H\dra \mathcal{H}$ be any
   hermitian structure. Then $\alpha\colon (\ov H\dra
   \mathcal{H})\dashrightarrow (\ov H\dra \mathcal{H}) $ is a tight
   isomorphism.
 \end{theorem}
 \begin{proof}
   First of all, we claim that if
   $\gamma:\ov{\mathcal{B}}\dra\ov{\mathcal{H}}$ is any isomorphism,
   then $\gamma^{-1}\circ\alpha\circ\gamma$ is tight if, and only if,
   $\alpha$ is tight. Indeed, denote by $\ov G\dra\mathcal{B}$ a
   representative of the hermitian structure on
   $\ov{\mathcal{B}}$. Then there is a diagram
   \begin{displaymath}
     \xymatrix{
       &	&	&\ov
       R\ar[ld]^{t_{1}}_{\sim}\ar[rd]_{t_{2}}^{\sim}	&	&\\
       &	&\ov P\ar[ld]^{w_{1}}_{\sim}\ar[rd]_{w_{2}}^{\sim}
       &	&\ov Q\ar[ld]^{w_{3}}_{\sim}\ar[rd]_{w_{4}}
       ^{\sim}&\\
       &\ov{G}'\ar[ld]_{\sim}^{u}\ar[rd]^{\sim}_{v}	&
       &\ov{H}'\ar[rd]^{\sim}_{f}\ar[ld]_{\sim}^{g}	&
       &\ov{G}'\ar[ld]_{\sim}^{v}\ar[rd]^{\sim}_{u}&\\
       \ov G\ar@{-->}[rr]	&	&\ov H\ar@{-->}[rr]	&
       &\ov H\ar@{-->}[rr]	&	&\ov G
     }
   \end{displaymath}
   for the liftings of $\gamma^{-1}$, $\alpha$, $\gamma$ to
   representatives, as well as for their composites, all whose squares
   are commutative up to homotopy. By acyclic calculus, we have the
   following chain of equalities
   \begin{multline*}
     [\ocone(u\circ w_{1}\circ t_{1})[1]]+[\ocone(u\circ w_{4}\circ
     t_{2})]=\\
     [\ocone(u)[1]]+[\ocone(v)]+[\ocone(g)[1]]+
     [\ocone(f)]+[\ocone(v)[1]]+[\ocone(u)]=\\
     [\ocone(g)[1]]+[\ocone(f)].
   \end{multline*}
   Thus, the right hand side vanishes if and only if the left hand
   side vanishes, proving the claim. This observation allows to reduce
   the proof of the lemma to the following situation: consider a
   diagram of complexes of hermitian vector bundles

   \begin{displaymath}
     \xymatrix{
       \overline{E}\ar[d]^{\Id}\ar[r]^{f}
       &\overline{F}\ar[d]^{\Id}\ar[r]^{\iota}
       &\overline{\cone}(f)\ar[r]^{\pi}\ar@{-->}[d]^{\phi}_{\sim} &
       \overline{E}[1]\ar[d]^{\Id}\ar[r]	&\dots\\
       \overline{E}\ar[r]^{f}	&\overline{F}\ar[r]^{\iota}
       &\overline{\cone}(f)\ar[r]^{\pi} &\overline{E}[1]\ar[r]
       &\dots,
     }
   \end{displaymath}
   which commutes in $\Db(X)$. We need to show that $\phi$ is a tight
   isomorphism. The commutativity of the diagram translates into the
   existence of bounded complexes of hermitian vector bundles
   $\overline{P}$ and $\overline{Q}$ and a diagram
   \begin{displaymath}
     \xymatrix{
       &						&
       &\overline{\cone}(f)\ar[rd]^{\pi}\ar@{-->}[dd]^{\phi}_{\sim}
       &\\
       \overline{F}\ar@/^/[rrru]^{\iota}\ar@/_/[rrrd]_{\iota}
       &\overline{P}\ar[l]_{\hspace{0.2cm} j}^{\sim}\ar[r]^{g}
       &\overline{Q} 	
       \ar[ru]_{u}^{\sim}\ar[rd]^{v}_{\sim}	&	&\overline{E}[1]\\
       &						&
       &\overline{\cone}(f)\ar[ru]_{\pi}	&
     }
   \end{displaymath}
   fulfilling the following properties: (a) $j$, $u$, $v$ are
   quasi-isomorphisms; (b) the squares formed by $\iota, j, g, u$ and
   $\iota, j, g, v$ are commutative up to homotopy; (c) the morphisms
   $u$, $v$ induce $\phi$ in the derived category. We deduce a
   commutative up to homotopy square
   \begin{displaymath}
     \xymatrix{
       \ocone(g)\ar[d]_{\tilde{v}}^{\sim}\ar[r]^{\tilde{u}}_{\sim}
       &\ocone(\iota)\ar[d]^{\tilde{\pi}}_{\sim}\\
       \ocone(\iota)\ar[r]^{\tilde{\pi}}_{\sim}		&\overline{E}[1].
     }
   \end{displaymath}
   The arrows $\tilde{u}$, $\tilde{v}$ are induced by $j,u$ and $j, v$
   respectively. Observe they are quasi-isomorphisms. Also the natural
   projection $\tilde{\pi}$ is a quasi-isomorphism. By acyclic
   calculus, we have
   \begin{displaymath}
     [\ocone(\tilde{\pi})]+[\ocone(\tilde{u})]=[\ocone(\tilde{\pi})]+[\ocone(\tilde{v})].
   \end{displaymath}
   Therefore we find
   \begin{equation}\label{eq:59}
     [\ocone(\tilde{u})]=[\ocone(\tilde{v})].
   \end{equation}
   Finally, notice there is an exact sequence
   \begin{displaymath}
     0\longrightarrow\ocone(u)\longrightarrow
     \ocone(\tilde{u})\longrightarrow \ocone(j[1]) \longrightarrow 0,
   \end{displaymath}
   whose rows are orthogonally split. Therefore,
   \begin{equation}\label{eq:60}
     [\ocone(\tilde{u})]=[\ocone(u)]+[\ocone(j[1])].
   \end{equation}
   Similarly we prove
   \begin{equation}\label{eq:61}
     [\ocone(\tilde{v})]=[\ocone(v)]+[\ocone(j[1])].
   \end{equation}
   From equations (\ref{eq:59})--(\ref{eq:61}) we infer
   \begin{displaymath}
     [\ocone(u)[1]]+[\ocone(v)]=0,
   \end{displaymath}
   as was to be shown.
 \end{proof}

 \begin{theorem} \label{thm:6} The object $\ov{C}(f)$ of equation
   \eqref{eq:her_cone} is well defined up to tight isomorphism.
 \end{theorem}
 \begin{proof}
   We first show the independence on the choice of $\ov E''$, up to
   tight isomorphism. To this end, it is enough to assume that there
   is a diagram
   \begin{displaymath}
     \xymatrix{
       && \ov E''' \ar[dl]_{\sim} \ar[rdd]&\\
       & \ov E''\ar[dl]_{\sim}\ar[rrd]&&\\
       \ov E &&&\ov E'
     }
   \end{displaymath}
   such that the triangle commutes up to homotopy. Fix such a homotopy. Then
   \begin{align*}
     [\ocone(\ocone(\ov E''',\ov E'),\ocone(\ov E'',\ov E'))]&=
     -[\ocone(E''', E'')],\\
     [\ocone(\ocone(\ov E''',\ov E),\ocone(\ov E'',\ov E))]&=
     -[\ocone(E''', E'')].
   \end{align*}
   By Lemma \ref{lemm:13}, the left hand sides of these relations
   agree and hence this implies that the hermitian structure does not
   depend on the choice of $\ov E''$.

   We now prove the independence on the choice of the representative
   $\ov E$. Let $\ov F\to \ov E$ be a tight morphism. Then we can
   construct a diagram
   \begin{displaymath}
     \xymatrix{
       & \ov E''' \ar[ddl]_{\sim} \ar[rd]^{\sim}&&\\
       && \ov E''\ar[dl]_{\sim}\ar[rd]&\\
       \ov F \ar[r]^{\sim} &\ov E &&\ov E',
     }
   \end{displaymath}
   where the square commutes up to homotopy. Choose one
   homotopy. Taking into account Lemma \ref{lemm:13}, we find
   \begin{align*}
     [\ocone(\ocone(\ov E''',\ov E'),\ocone(\ov E'',\ov E'))]&=
     -[\ocone(E''', E'')],\\
     [\ocone(\ocone(\ov E''',\ov F),\ocone(\ov E'',\ov E))]&=
     -[\ocone(E''', E'')]+[\ocone(\ov F,\ov
     E)]\\
     &=-[\ocone(E''', E'')].
   \end{align*}
   Hence the definitions of $\ov{C}(f)$ using $\ov E$ or $\ov F$ agree
   up to tight isomorphism.  The remaining possible choices of
   representatives are treated analogously.
 \end{proof}

\begin{remark}
  The construction of $\ocone(f)$ involves the choice of $\cone(f)$,
  which is unique up to isomorphism. Since the construction of
  $\ov{C}(f)$ \eqref{eq:her_cone} does not depend on the choice of
  $\cone(f)$, by Theorem \ref{thm:6bis}, we see that different
  choices of $\cone(f)$ give rise to tightly isomorphic hermitian
  cones. Therefore $\ocone(f)$ is well defined up to tight isomorphism
  and we will usually call it \emph{the}
  hermitian cone of $f$. When the morphism is clear, we will also write
  $\ocone(\ov{\mathcal{F}},\ov{\mathcal{G}})$ to refer to it.
\end{remark}

The hermitian cone satisfies the same relations than the usual cone.

\begin{proposition} \label{prop:10} Let $f\colon \ov {\mathcal{F}}\dra
  \ov {\mathcal{G}}$ be a morphism in $\oDb(X)$. Then, the natural
  morphisms
  \begin{gather*}
    \ocone(\ov{\mathcal{G}},\ocone(f))\dra
    \ov{\mathcal{F}}[1],\\
    \ov{\mathcal{G}}\dra \ocone(\ocone(f)[-1],\ov{\mathcal{F}})
  \end{gather*}
  are tight isomorphisms.
\end{proposition}
\begin{proof}
  After choosing representatives, there are isometries
  \begin{align*}
    \ocone(\ocone(\ov{\mathcal{G}},\ocone(f)),
    \ov{\mathcal{F}}[1])\cong &
    \ocone(\ocone(\Id_{\mathcal{F}}),\ocone(\Id_{\mathcal{G}}))\cong\\
    &\ocone(\ov{\mathcal{G}}, \ocone(\ocone(f)[-1],\ov{\mathcal{F}})).
  \end{align*}
  Since the middle term is meager, the same is true for the other two.
\end{proof}

We next extend some basic constructions in $\Db(X)$ to $\oDb(X)$.

\noindent\textbf{Derived tensor product.}  Let
$\overline{\mathcal{F}}_{i}=(
\overline{E}_{i}\dashrightarrow\mathcal{F}_{i})$, $i=1,2$, be objects
of $\oDb(X)$. The derived tensor product $\mathcal{F}_{1}\Lotimes
\mathcal{F}_{2}$ is endowed with a natural hermitian structure
\begin{equation}
  \label{eq:29}
  \overline{E}_{1}\otimes\overline{E}_{2}\dashrightarrow
  \mathcal{F}_{1}\Lotimes \mathcal{F}_{2},
\end{equation}
that is well defined by Theorem \ref{thm:4}~\ref{item:24}. We write
$\overline{\mathcal{F}}_{1}\Lotimes \overline{\mathcal{F}}_{2}$ for
the resulting object in $\oDb(X)$.

\noindent\textbf{Derived internal $\Hom$ and dual objects.}  Let
$\overline{\mathcal{F}}_{i}=(
\overline{E}_{i}\dashrightarrow\mathcal{F}_{i})$, $i=1,2$, be objects
of $\oDb(X)$. The derived internal $\Hom$, $
\RuHom(\mathcal{F}_{1},\mathcal{F}_{2})$ is endowed with a natural
hermitian structure
\begin{equation}
  \label{eq:22}
  \uHom(\overline{E}_{1},\overline{E}_{2})\dashrightarrow
  \RuHom(\mathcal{F}_{1},\mathcal{F}_{2}),
\end{equation}
that is well defined by Theorem \ref{thm:4}~\ref{item:24}. We write
$\RuHom(\ov{\mathcal{F}}_{1},\ov{\mathcal{F}}_{2})$ for the resulting
object in $\oDb(X)$.

In particular, denote by $\ov {\mathcal{O}}_{X}$ the structural sheaf
with the metric $\|1\|=1$. Then, for every object
$\ov{\mathcal{F}} \in \oDb(X)$, the \emph{dual object} is defined to
be
\begin{equation}
  \label{eq:30}
  \ov{\mathcal{F}}^{\vee}=\RuHom(\ov
  {\mathcal{F}},\ov{\mathcal{O}}_{X}).
\end{equation}

 \noindent\textbf{Left derived inverse image.} Let
 $g\colon X^{\prime}\rightarrow X$ be a morphism of smooth algebraic 
 varieties over $\CC$ and $\ov{\mathcal{F}}=(\ov E\dra \mathcal{F})\in
 \Ob \oDb(X)$. Then the left derived inverse image $\Ld
 g^{\ast}(\mathcal{F})$ is equipped with the hermitian structure
 $g^{\ast}(\overline{E})\dashrightarrow\Ld g^{\ast}(\mathcal{F})$,
 that is well defined up to tight isomorphism by Theorem
 \ref{thm:4}~\ref{item:25}. As it is customary, we will pretend that
 $\Ld g^{\ast}$ is a functor. The notation for the corresponding
 object in $\oDb(X^{\prime})$ is $\Ld
 g^{\ast}(\overline{\mathcal{F}})$. If $f\colon
 \overline{\mathcal{F}}_{1}\dashrightarrow \overline{\mathcal{F}}_{2}$
 is a morphism in $\oDb(X)$, we denote by $\Ld g^{\ast}(f)\colon \Ld
 g^{\ast}(\overline{\mathcal{F}}_{1})\dashrightarrow\Ld g^{\ast}
 (\overline{\mathcal{F}}_{2})$ its left derived inverse image by $g$.

 The functor $\Ld g^{\ast}$ preserves the structure of principal
 fibered category with flat connection and the formation of hermitian
 cones. Namely we have the following result that is easily proved.

 \begin{theorem} \label{thm:9} Let $g\colon X^{\prime}\rightarrow X$
   be a morphism of smooth algebraic varieties over $\CC$ and let
   $f\colon \ov {\mathcal{F}}_{1}\dra \ov {\mathcal{F}}_{2}$ be a
   morphism in $\oDb(X)$.
   \begin{enumerate}
   \item The functor $\Ld g^{\ast}$ preserves the forgetful functor:
     \begin{displaymath}
       \mathfrak{F}\circ \Ld g^{\ast}=\Ld g^{\ast}\circ \mathfrak{F}
     \end{displaymath}
   \item The restriction $\Ld g^{\ast}\colon\KA(X)\to \KA(X')$ is a
     group homomorphism.
   \item The functor $\Ld g^{\ast}$ is equivariant with respect to the
     actions of $\KA(X)$ and $\KA(X')$.
   \item The functor $\Ld g^{\ast}$ preserves parallel transport: if
     $f$ is an isomorphism, then
     \begin{displaymath}
       \Ld g^{\ast}\circ \mathfrak{t}_{f}=\mathfrak{t}_{\Ld
         g^{\ast}(f)}\circ \Ld g^{\ast}.
     \end{displaymath}
   \item The functor $\Ld g^{\ast}$ preserves hermitian cones:
     \begin{displaymath}
       \Ld g^{\ast}(\ocone(f))=\ocone(\Ld g^{\ast}(f)).
     \end{displaymath}
   \end{enumerate}
 \end{theorem}
 \hfill $\square$

 \noindent\textbf{Classes of isomorphisms and distinguished
   triangles.}
 Let $f\colon \ov {\mathcal{F}}\overset{\sim}{\dashrightarrow}
 \ov{\mathcal{G}}$ be an isomorphism in $\oDb(X)$. To it, we attach a
 class $[f]\in \KA(X)$ that measures the default of being a tight
 isomorphism. This class is defined using the hermitian cone.
 \begin{equation}
   \label{eq:57}
   [f]=[\ocone(f)].
 \end{equation}
 Observe the abuse of notation: we wrote $[\ocone(f)]$ for the class
 in $\KA(X)$ of the hermitian structure of a hermitian cone of
 $f$. This is well defined, since the hermitian cone is unique up to
 tight isomorphism. Alternatively, we can construct $[f]$ using
 parallel transport as follows.  There is a unique element $\ov
 A\in\KA(X)$ such that
 \begin{displaymath}
   \ov {\mathcal{G}}=\mathfrak{t}_{f}\ov {\mathcal{F}}+\ov A.
 \end{displaymath}
 We denote this element by $\ov {\mathcal{G}}-\mathfrak{t}_{f}\ov
 {\mathcal{F}}$. Then
 \begin{displaymath}
   [f]=\ov {\mathcal{G}}-\mathfrak{t}_{f}\ov
   {\mathcal{F}}.
 \end{displaymath}
 By the very definition of parallel transport, both definitions
 clearly agree.

 \begin{definition}\label{def:6}
   A \textit{distinguished triangle in} $\oDb(X)$, consists in a
   diagram
   \begin{equation}\label{eq:64}
     \overline{\tau}=(u,v,w):
     \overline{\mathcal{F}}\overset{u}{\dashrightarrow}
     \overline{\mathcal{G}}
     \overset{v}{\dashrightarrow}\overline{\mathcal{H}}
     \overset{w}{\dashrightarrow}\overline{\mathcal{F}}[1]
     \overset{u}{\dashrightarrow}\dots
   \end{equation}
   in $\oDb(X)$, whose underlying morphisms in $\Db(X)$ form a
   distinguished triangle. We will say that it is \emph{tightly
     distinguished} if there is a commutative diagram
   \begin{equation}\label{eq:63}
     \xymatrix{
       \ov{\mathcal{F}}\ar@{-->}[r]\ar[d]^{\Id}	
       &\ov{\mathcal{G}}\ar@{-->}[r]\ar[d]^{\Id}
       &\ocone(\ov {\mathcal{F}},\ov{\mathcal{G}})
       \ar@{-->}[r]\ar@{-->}[d]^{\alpha }
       & \ov{\mathcal{F}}[1]\ar@{-->}[r]	\ar[d]^{\Id}
       &\dots\\
       \ov{\mathcal{F}}\ar@{-->}[r]	
       &\ov{\mathcal{G}}\ar@{-->}[r]
       &\ov{\mathcal{H}}\ar@{-->}[r]
       & \ov{\mathcal{F}}[1]\ar@{-->}[r]
       &\dots,
     }
   \end{equation}
   with $\alpha $ a tight isomorphism.
 \end{definition}

 To every distinguished triangle in $\oDb(X)$ we can associate a class
 in $\KA(X)$ that measures the default of being tightly distinguished.
 Let $\ov {\tau }$ be a distinguished triangle as in
 \eqref{eq:64}. Then there is a diagram as \eqref{eq:63}, but with
 $\alpha $ an isomorphism non-necessarily tight. Then we define
 \begin{equation}
   \label{eq:65}
   [\ov{\tau }]=[\alpha ].
 \end{equation}
 By Theorem \ref{thm:6bis}, the class $[\alpha]$ does not depend on
 the particular choice of morphism $\alpha$ in $\oDb(X)$ for which
 \eqref{eq:63} commutes. Hence \eqref{eq:65} only depends on
 $\ov{\tau}$.

 \begin{theorem}\label{thm:10}\ 
   \begin{enumerate}
   \item \label{item:51} Let $f$ be an isomorphism in $\oDb(X)$
     (respectively $\ov \tau $ a distinguished triangle). Then $[f]=0$
     (respectively $[\ov \tau ]=0$) if and only if $f$ is a tight isomorphism
     (respectively $\ov \tau $ is tightly distinguished).
   \item \label{item:50} Let $g\colon X'\to X$ be a morphism of smooth
     complex varieties, let $f$ be an isomorphism in $\oDb(X)$ and
     $\ov \tau $ a distinguished triangle in $\oDb(X)$. Then
     \begin{displaymath}
       \Ld g^{\ast}[f]=[\Ld g^{\ast}f],\qquad \Ld g^{\ast}[\ov \tau
       ]=[\Ld g^{\ast}\ov \tau ].
     \end{displaymath}
     In particular, tight isomorphisms and tightly distinguished
     triangles are preserved under left derived inverse images.
   \item \label{item:16}%
     Let $f\colon
     \overline{\mathcal{F}}\dashrightarrow\overline{\mathcal{G}}$ and
     $h\colon \overline{\mathcal{G}}\dashrightarrow
     \overline{\mathcal{H}}$ be two isomorphisms in $\oDb(X)$. Then:
     \begin{displaymath}
       [h\circ f]=[h]+[f].
     \end{displaymath}
     In particular, $[f^{-1}]=-[f]$.
   \item \label{item:17} For any distinguished triangle $\ov \tau $ in
     $\oDb(X)$ as in Definition \ref{def:6}, the rotated triangle
     \begin{displaymath}
       \overline{\tau}'\colon\
       \overline{\mathcal{G}}\overset{v}{\dashrightarrow}\overline{\mathcal{H}}
       \overset{w}{\dashrightarrow}\overline{\mathcal{F}}[1]
       \overset{-u[1]}{\dashrightarrow}\overline{\mathcal{G}}[1]
       \overset{v[1]}{\dashrightarrow}\dots
     \end{displaymath}
     satisfies
     \begin{math}
       [\ov \tau ']=-[\ov \tau ].
     \end{math}
     In particular, rotating preserves tightly distinguished
     triangles.
   \item \label{item:28} For any acyclic complex $\ov{\mathcal{F}}$,
     we have
     \begin{displaymath}
       [\overline{\mathcal{F}}\to 0\to
       0\to\dots]=
       [\ov{\mathcal{F}}].
     \end{displaymath}
   \item \label{item:47} If
     $f\colon\ov{\mathcal{F}}\dra\ov{\mathcal{G}}$ is an isomorphism
     in $\oDb(X)$, then
     \begin{displaymath}
       [0\to\overline{\mathcal{F}}\dra\ov{\mathcal{G}}
       \to\dots]=[f].
     \end{displaymath}
   \item \label{item:48} For a commutative diagram of distinguished
     triangles
     \begin{displaymath}
       \xymatrix{
         \ov \tau \ar@{-->}[d]
         &\ov{\mathcal{F}}\ar@{-->}[r]\ar@{-->}[d]^{f}_{\sim}	
         &\ov{\mathcal{G}}\ar@{-->}[r]\ar@{-->}[d]^{g}_{\sim}
         &\ov{\mathcal{H}}
         \ar@{-->}[r]\ar@{-->}[d]^{h} _{\sim}
         & \ov{\mathcal{F}}[1]\ar@{-->}[r]	\ar@{-->}[d]^{f[1]}_{\sim}
         &\dots\\
         \ov \tau ' &\ov{\mathcal{F}}'\ar@{-->}[r]	
         &\ov{\mathcal{G}}'\ar@{-->}[r]
         &\ov{\mathcal{H}}'\ar@{-->}[r]
         & \ov{\mathcal{F}}'[1]\ar@{-->}[r]
         &\dots,
       }
     \end{displaymath}
     the following relation holds:
     \begin{displaymath}
       [\ov \tau ']
       -[\ov \tau ]=
       [f]-[g]+[h].
     \end{displaymath}
   \item \label{item:49} For a commutative diagram of distinguished
     triangles
     \begin{equation}\label{eq:66}
       \xymatrix{
         \ov \tau \ar@{-->}[d]
         &\ov{\mathcal{F}}\ar@{-->}[r]\ar@{-->}[d]	
         &\ov{\mathcal{G}}\ar@{-->}[r]\ar@{-->}[d]
         &\ov{\mathcal{H}}
         \ar@{-->}[r]\ar@{-->}[d]
         & \ov{\mathcal{F}}[1]\ar@{-->}[r]	\ar@{-->}[d]
         &\dots\\
         \ov \tau' \ar@{-->}[d]
         &\ov{\mathcal{F}}'\ar@{-->}[r]\ar@{-->}[d]
         &\ov{\mathcal{G}}'\ar@{-->}[r]\ar@{-->}[d]
         &\ov{\mathcal{H}}'
         \ar@{-->}[r]\ar@{-->}[d]
         & \ov{\mathcal{F}}'[1]\ar@{-->}[r]	\ar@{-->}[d]
         &\dots\\
         \ov \tau '' &\ov{\mathcal{F}}''\ar@{-->}[r]\ar@{-->}[d]	
         &\ov{\mathcal{G}}''\ar@{-->}[r]  \ar@{-->}[d]
         &\ov{\mathcal{H}}''\ar@{-->}[r]  \ar@{-->}[d]
         & \ov{\mathcal{F}}''[1]\ar@{-->}[r] \ar@{-->}[d]
         &\dots\\
         &\ov{\mathcal{F}}[1]\ar@{-->}[r]\ar@{-->}[d]	
         &\ov{\mathcal{G}}[1]\ar@{-->}[r]\ar@{-->}[d]
         &\ov{\mathcal{H}}[1]
         \ar@{-->}[r]\ar@{-->}[d] &\ov{\mathcal{F}}[2]\ar@{-->}[r]
         \ar@{-->}[d]
         &\dots&\\
         &\vdots &\vdots &\vdots & \vdots &&\\
         & \ov \eta \ar@{-->}[r] & \ov \eta' \ar@{-->}[r]  & \ov \eta''
         &&&
       }
     \end{equation}
     the following relation holds:
     \begin{displaymath}
       [\ov \tau ]-[\ov \tau' ]
       +[\ov \tau'' ]=
       [\ov \eta ]-[\ov \eta' ]
       +[\ov \eta'' ].
     \end{displaymath}
   \end{enumerate}
 \end{theorem}
 \begin{proof}
   The first two statements are clear.  For the third, we may assume
   that $f$ and $g$ are realized by quasi-isomorphisms
   \begin{displaymath}
     f\colon \ov F\longrightarrow \ov G,
     \quad g\colon  \ov G\longrightarrow \ov H.
   \end{displaymath}
   Then the result follows from Theorem \ref{thm:7}~\ref{item:41}.
   The fourth assertion is a consequence of Proposition
   \ref{prop:10}. Then \ref{item:28}, \ref{item:47} and \ref{item:48}
   follow from equation \eqref{eq:65} and the fourth statement. The
   last property is derived from \ref{item:48} by comparing the
   diagram to a diagram of tightly distinguished triangles.
 \end{proof}

 As an application of the class in $\KA(X)$ attached to a
 distinguished triangle, we exhibit a natural morphism
 $K_{1}(X)\to\KA(X)$. This is included for the sake of completeness,
 but won't be needed in the sequel.

\begin{proposition}\label{prop:K1_to_KA}
  There is a natural morphism of groups $K_{1}(X)\to\KA(X)$.
\end{proposition}
\begin{proof}
  We follow the definitions and notations of \cite{Burgos-Wang}. From
  \emph{loc. cit.} we know it is enough to construct a morphism of
  groups
  \begin{equation}\label{eq:K1-KA}
    H_{1}(\widetilde{\mathbb{Z}}C(X))\to\KA(X).
  \end{equation}
  By definition, the piece of degree $n$ of the homological complex
  $\widetilde{\mathbb{Z}}C(X)$ is
  \begin{displaymath}
    \widetilde{\mathbb{Z}}C_{n}(X)=\mathbb{Z}C_{n}(X)/D_{n}.
  \end{displaymath}
  Here $\mathbb{Z}C_{n}(X)$ stands for the free abelian group on
  metrized exact $n$-cubes and $D_{n}$ is the subgroup of degenerate
  elements. A metrized exact $1$-cube is a short exact sequence of
  hermitian vector bundles. Hence, for such a $1$-cube
  $\ov{\varepsilon}$, there is a well defined class in
  $\KA(X)$. Observe that this class coincides with the class of
  $\ov{\varepsilon}$ thought as a distinguished triangle in
  $\oDb(X)$. Because $\KA(X)$ is an abelian group, it follows the
  existence of a morphism of groups
  \begin{displaymath}
    \mathbb{Z}C_{1}(X)\longrightarrow\KA(X).
  \end{displaymath}
  From the definition of degenerate cube \cite[Def. 3.3]{Burgos-Wang}
  and the construction of $\KA(X)$, this morphism clearly factors
  through $\widetilde{\mathbb{Z}}C_{1}(X)$. The definition of the
  differential $d$ of the complex $\widetilde{\mathbb{Z}}C(X)$
  \cite[(3.2)]{Burgos-Wang} and Theorem \ref{thm:10} \ref{item:49}
  ensure that $d\mathbb{Z}C_{2}(X)$ is in the kernel of the
  morphism. Hence we derive the existence of a morphism
  \eqref{eq:K1-KA}.
\end{proof}

\noindent\textbf{Classes of complexes and of direct images of
  complexes.} In
\cite[Section 2]{BurgosLitcanu:SingularBC} the notion of homological
exact sequences of metrized coherent sheaves is treated. In the
present article, this situation will arise in later
considerations. Therefore we provide the link between the point of
view of \emph{loc. cit.} and the formalism adopted here. The reader
will find no difficulty  to translate it to cohomological complexes.

Consider a homological complex
\begin{displaymath}
  \overline{\varepsilon }:\quad
  0\to \overline{\mathcal{F}}_{m}\to \dots \to
  \overline{\mathcal{F}}_{l}\to 0
\end{displaymath}
of metrized coherent sheaves, namely coherent sheaves provided with
hermitian structures $\ov {\mathcal{F}}_{i}=(\mathcal{F}_{i},\ov
F_{i}\dra \mathcal{F}_{i})$. We may equivalently see
$\overline{\varepsilon}$ as a cohomological complex, by the usual
relabeling $\ov{\mathcal{F}}^{-i}=\ov{\mathcal{F}}_{i}$. This will be
freely used in the sequel, especially in cone constructions.
\begin{definition}\label{def:1}
  The complex $\ov \varepsilon $ defines an object $[\ov \varepsilon
  ]\in \Ob \oDb(X)$ that is determined inductively by the condition
  \begin{displaymath}
    [\ov \varepsilon ]=\ocone(\ov{\mathcal{F }}_{m}[m],[\sigma _{<m}\ov
    \varepsilon ]).
  \end{displaymath}
  Here $\sigma _{<m}$ is the homological b\^ete filtration and
  $\ov{\mathcal{F
    }}_{m}$ denotes a cohomological complex concentrated in degree
  zero. Hence, $\ov{\mathcal{F }}_{m}[m]$ is a cohomological complex
  concentrated in degree $-m$.
\end{definition}

If $\ov{E}$ is a hermitian vector bundle on $X$, then
\begin{math}
  [\ov{\varepsilon}\otimes \ov{E}]=[\ov{\varepsilon}]\otimes \ov{E}
\end{math}.
According to Definition \ref{def:category_oDb}, if $\varepsilon $ is
an acyclic complex, then we also have the corresponding class
$[[\ov{\varepsilon}]]$ in $\KA(X)$. We will employ the lighter
notation $[\ov{\varepsilon}]$ for this class.

Given a morphism $\varphi\colon \ov{\varepsilon}\to\ov{\mu}$ of
bounded complexes of metrized coherent sheaves, the pieces of the
complex $\cone(\varepsilon,\mu)$ are naturally endowed with hermitian
metrics. We thus get a complex of metrized coherent sheaves
$\ov{\cone(\varepsilon,\mu)}$. Hence Definition \ref{def:1} provides
an object $[\ov{\cone(\varepsilon,\mu)}]$ in $\oDb(X)$. On the other
hand, Definition \ref{def:her_cone} attaches to $\varphi$ the
hermitian cone $\ocone([\ov{\varepsilon}],[\ov{\mu}])$, which is well
defined up to tight isomorphism. Both constructions actually agree.
\begin{lemma}\label{lemm:3}
  Let $\ov{\varepsilon}\to\ov{\mu}$ be a morphism of bounded complexes
  of metrized coherent sheaves on $X$. Then there is a tight
  isomorphism
  \begin{displaymath}
    \ocone([\ov{\varepsilon}],[\ov{\mu}])\cong [\ov{\cone(\varepsilon,\mu)}],
  \end{displaymath}
\end{lemma}
\begin{proof}
  The case when $\varepsilon$ and $\mu$ are both concentrated in a
  single degree $d$ is clear. The general case follows by induction
  taking into account Definition \ref{def:1}.
\end{proof}

Assume now that $f\colon X\to Y$ is a morphism of smooth complex
varieties and, for each complex $\Rd f_{\ast}\mathcal{F}_{i}$, we have
chosen a hermitian structure $\ov{\Rd f_{\ast} \mathcal{F}_{i}}=(\ov
E_{i}\dra \Rd f_{\ast} \mathcal{F}_{i})$. Denote by $\ov {\Rd
  f_{\ast}\varepsilon }$ this choice of metrics.
\begin{definition} \label{def:5} The family of hermitian structures
  $\ov {\Rd f_{\ast}\varepsilon }$ defines an object $[\ov{\Rd
    f_{\ast}\varepsilon }]\in \Ob \oDb(Y)$ that is determined
  inductively by the condition
  \begin{displaymath}
    [\ov {\Rd f_{\ast}\varepsilon }]=\ocone (\ov {\Rd f_{\ast}
      \mathcal{F}_{m}}[m],[\ov {\Rd f_{\ast}\sigma _{< m} \varepsilon} ]).
  \end{displaymath}
\end{definition}

We remark that the notation $\ov {\Rd
  f_{\ast}\varepsilon }$ means that the hermitian structure is
chosen after taking the direct image and it is not determined by the
hermitian structure on $\ov \varepsilon $.

If $\ov{F}$ is a hermitian vector bundle on $Y$, then the
object $[\ov{\Rd f_{\ast}(\varepsilon\otimes f^{\ast} F)}]$ (whose
definition is obvious)
satisfies
\begin{displaymath}
  [\ov{\Rd f_{\ast}(\varepsilon\otimes f^{\ast} F)}]
  =[\ov{\Rd f_{\ast}\varepsilon}]\otimes\ov{F}.
\end{displaymath}
Notice also that if $\varepsilon$ is an acyclic complex on $X$,
we have the class $[\ov {\Rd f_{\ast}\varepsilon}]\in \KA(Y)$.

Let $\varepsilon\to\mu$ be a morphism of bounded complexes of coherent
sheaves on $X$ and $f\colon X\to Y$ a morphism of smooth complex
varieties. Fix choices of metrics $\ov{\Rd f_{\ast}\varepsilon}$ and
$\ov{\Rd f_{\ast}\mu}$. Then there is an obvious choice of metrics on
$\Rd f_{\ast}\cone(\varepsilon,\mu)$, that we denote $\ov{\Rd
  f_{\ast}\cone(\varepsilon,\mu)}$, and hence an object $[\ov{\Rd
  f_{\ast}\cone(\varepsilon,\mu)}]$ in $\oDb(Y)$. On the other hand,
we also have the hermitian cone $\ocone([\ov{\Rd
  f_{\ast}\varepsilon}],[\ov{\Rd f_{\ast}\mu}])$. Again both
definitions agree.

\begin{lemma}\label{lemm:3bis}
  Let $\varepsilon\to\mu$ be a morphism of bounded complexes of
  coherent sheaves on $X$ and $f\colon X\to Y$ a morphism of smooth
  complex varieties. Assume that families of metrics $\ov{\Rd
    f_{\ast}\varepsilon}$ and $\ov{\Rd f_{\ast}\mu}$ are chosen. Then there is a
  tight isomorphism
  \begin{displaymath}
    \ocone([\ov{\Rd f_{\ast}\varepsilon}],[\ov{\Rd f_{\ast}\mu}])\cong [\ov{\Rd f_{\ast}\cone(\varepsilon,\mu)}].
  \end{displaymath}
\end{lemma}
\begin{proof}
  If $\varepsilon$ and $\mu$ are concentrated in a single degree $d$,
  then the statement is obvious. The proof follows by induction and
  Definition \ref{def:5}.
\end{proof}

The objects we have defined are compatible with short exact sequences,
in the sense of the following statement.
\begin{proposition} \label{prop:4} Consider a commutative diagram of
  exact sequences of coherent sheaves on $X$
  \begin{displaymath}
    \xymatrix{
      & & &0\ar[d] & &0\ar[d] & \\
      &\mu' &0\ar[r] &\mathcal{F}_{m}'\ar[r]\ar[d] &\dots\ar[r] &\mathcal{F}_{l}'\ar[r]\ar[d] &0
      \\
      &\mu &0\ar[r] &\mathcal{F}_{m}\ar[r]\ar[d] &\dots\ar[r] &\mathcal{F}_{l}\ar[r]\ar[d] &0
      \\
      &\mu'' &0\ar[r] &\mathcal{F}_{m}''\ar[r]\ar[d] &\dots\ar[r] &\mathcal{F}_{l}''\ar[r]\ar[d] &0
      \\
      & & &0 & &0 & \\
      & & &\xi_{m} &\dots &\xi_{l}. &
    }
  \end{displaymath}
  Let $f\colon X\to Y$ be a morphism of smooth complex varieties and
  choose hermitian structures on the sheaves $\mathcal{F}_{j}'$,
  $\mathcal{F}_{j}$, $\mathcal{F}_{j}''$ and on the objects $\Rd
  f_{\ast}\mathcal{F}_{j}'$, $\Rd f_{\ast}\mathcal{F}_{j}$ and $\Rd
  f_{\ast}\mathcal{F}_{j}''$, $j=l,\dots, m$. Then the following
  equalities hold in $\KA(X)$ and $\KA(Y)$, respectively:
  \begin{align*}
    &\sum_{j}(-1)^{j}[\ov{\xi}_{j}]=[\ov{\mu}']-[\ov{\mu}]+[\ov{\mu}''],\\
    &\sum_{j}(-1)^{j}[\ov{\Rd f_{\ast}\xi}_{j}]=[\ov{\Rd
      f_{\ast}\mu}']-[\ov{\Rd f_{\ast} \mu}]+[\ov{\Rd f_{\ast}\mu}''].
  \end{align*}
\end{proposition}
\begin{proof}
  The lemma follows inductively taking into account definitions
  \ref{def:1} and \ref{def:5} and Theorem \ref{thm:10}~\ref{item:49}.
\end{proof}
\begin{corollary}\label{cor:5}
  Let $\ov{\varepsilon}\to\ov{\mu} $ be a morphism of exact sequences
  of metrized coherent sheaves. Let $f\colon X\to Y$ be a morphism of
  smooth complex varieties and fix families of metrics $\ov{\Rd
    f_{\ast}\varepsilon}$ and $\ov{\Rd f_{\ast}\mu}$. Then the following
  equalities in $\KA(X)$ and $\KA(Y)$, respectively, hold
  \begin{align}
    &[\ov{\cone(\varepsilon,\mu)}]=[\ov{\mu}]-[\ov{\varepsilon}],\\
    &[\ov{\Rd f_{\ast}\cone(\varepsilon,\mu})]=[\ov{\Rd
      f_{\ast}\mu}]-[\ov{\Rd f_{\ast}\varepsilon}].
  \end{align}
\end{corollary}
\begin{proof}
  The result readily follows from lemmas \ref{lemm:3}, \ref{lemm:3bis}
  and Proposition \ref{prop:4}.
\end{proof}

\noindent\textbf{Hermitian structures on cohomology.}
Let $\mathcal{F}$ be an object of $\Db(X)$ and denote by $\mathcal{H}$
its cohomology complex. Observe that $\mathcal{H}$ is a bounded
complex with 0 differentials. By the preceding discussion and because
$\KA(X)$ acts transitively on hermitian structures, giving a hermitian
structure on $\mathcal{H}$ amounts to give hermitian structures on the
individual pieces $\mathcal{H}^{i}$. We show that to these data there
is attached a natural hermitian structure on the complex
$\mathcal{F}$. This situation will arise when considering cohomology
sheaves endowed with $L^2$ metric structures. The construction is
recursive. If the cohomology complex is trivial, then we endow
$\mathcal{F}$ with the trivial hermitian structure. Otherwise, let
$\mathcal{H}^{m}$ be the highest non-zero cohomology sheaf. The
canonical filtration $\tau ^{\leq m}$ is given by
\begin{displaymath}
  \tau^{\leq
    m}\mathcal{F}\colon\quad\dots\to\mathcal{F}^{m-2}\to\mathcal{F}^{m-1}
  \to\ker(\dd^{m})\to  0.
\end{displaymath}
By the condition on the highest non vanishing cohomology sheaf, the
natural inclusion is a quasi-isomorphism:
\begin{equation}\label{eq:her_coh_1}
  \tau^{\leq m}\mathcal{F}\overset{\sim}{\longrightarrow}\mathcal{F}.
\end{equation}
We also introduce the subcomplex
\begin{displaymath}
  \widetilde{\mathcal{F}}\colon\quad\dots\to\mathcal{F}^{m-2}\to
  \mathcal{F}^{m-1}\to\Im(\dd^{m-1})\to 0.
\end{displaymath}
Observe that the cohomology complex of $\widetilde{\mathcal{F}}$ is
the b\^ete truncation $\mathcal{H}/\sigma^{\ge m}\mathcal{H}$. By induction,
$\widetilde{\mathcal{F}}$ carries an induced hermitian
structure. We also have an exact sequence
\begin{equation}\label{eq:her_cor_2}
  0\to\widetilde{\mathcal{F}}\to\tau^{\leq m}\mathcal{F}\to\mathcal{H}^{m}[-m]\to 0.
\end{equation}
Taking into account the quasi-isomorphism \eqref{eq:her_coh_1} and the
exact sequence \eqref{eq:her_cor_2}, we construct a natural
commutative diagram of distinguished triangles in $\Db(X)$
\begin{displaymath}
  \xymatrix{
    \mathcal{H}^{m}[-m-1]\ar@{-->}[r]^{\hspace{0.6cm} 0}\ar[d]^{\Id}	 &\widetilde{\mathcal{F}}\ar@{-->}[r]\ar[d]^{\Id}
    &\mathcal{F}\ar@{-->}[r]\ar@{-->}[d]^{\sim}	&\mathcal{H}^{m}[m]\ar[d]^{\Id}\\
    \mathcal{H}^{m}[-m-1]\ar@{-->}[r]^{\hspace{0.6cm} 0}	&\widetilde{\mathcal{F}}\ar@{-->}[r]
    &\cone(\mathcal{H}^{m}[-m-1],\widetilde{\mathcal{F}})\ar@{-->}[r]	&\mathcal{H}^{m}[m].
  }
\end{displaymath}
By the hermitian cone construction and Theorem \ref{thm:6bis}, we see
that hermitian structures on $\widetilde{\mathcal{F}}$ and
$\mathcal{H}^{m}$ induce a well defined hermitian structure on
$\mathcal{F}$.

\begin{definition}\label{def:her_coh}
  Let $\mathcal{F}$ be an object of $\Db(X)$ with cohomology complex
  $\mathcal{H}$. Assume the pieces $\mathcal{H}^{i}$ are endowed with
  hermitian structures. The hermitian structure on $\mathcal{F}$
  constructed above will be called the \emph{hermitian structure induced by
  the hermitian structure on the cohomology complex} and will be
denoted $(\mathcal{F},\ov{\mathcal{H}})$.
\end{definition}

The following proposition is a direct consequence of the definitions.

\begin{proposition}\label{prop:her_coh}
  Let $\varphi\colon\mathcal{F}_{1}\dra\mathcal{F}_{2}$ be an
  isomorphism in $\Db(X)$. Assume the pieces of the cohomology
  complexes $\mathcal{H}_{1}$, $\mathcal{H}_{2}$ of $\mathcal{F}_{1}$,
  $\mathcal{F}_{2}$ are endowed with hermitian structures. If the
  induced isomorphism in cohomology
  $\varphi_{\ast}\colon\mathcal{H}_{1}\to\mathcal{H}_{2}$ is tight,
  then $\varphi$ is tight for the induced hermitian structures on
  $\mathcal{F}_{1}$ and $\mathcal{F}_{2}$.
\end{proposition}
\hfill $\square$

\section{Bott-Chern classes for
isomorphisms and distinguished triangles in $\oDb(X)$}
\label{sec:bott-chern-classes}

In this section we will define Bott-Chern classes for
isomorphisms and distinguished triangles in $\oDb(X)$.
The natural context where one can define the Bott-Chern
classes is that of Deligne
complexes. For details about Deligne complexes the reader is referred
to \cite{Burgos:CDB} and \cite{BurgosKramerKuehn:cacg}. In this
section we will use the same notations as in
\cite{BurgosLitcanu:SingularBC} \S1. In particular,
the \emph{Deligne algebra of differential forms} on $X$
is denoted by
\begin{math}
  \mathcal{D}^{\ast}(X,\ast).
\end{math}
and we use the notation
\begin{displaymath}
  \widetilde
  {\mathcal{D}}^{n}(X,p)=\left. \mathcal{D}^{n}(X,p)\right/
  \dd_{\mathcal{D}}\mathcal{D}^{n-1}(X,p).
\end{displaymath}

When characterizing axiomatically Bott-Chern classes, the
basic tool to exploit the functoriality axiom is to use a
deformation  parametrized by $\PP^{1}$. This argument leads to the
following lemma that will be used to prove the uniqueness of the
Bott-Chern classes introduced in this section.

\begin{lemma} \label{lemm:1}
  Let $X$ be a smooth complex variety. Let
  $\widetilde \varphi$ be an assignment that, to each smooth morphism
  of complex
  varieties $g\colon X'\to X$ and each acyclic complex $\ov A$ of
  hermitian vector bundles on $X'$
  assigns a class
  \begin{displaymath}
   \widetilde \varphi(\ov A)\in \bigoplus
  _{n,p}\widetilde{\mathcal{D}}^{n}(X',p)
  \end{displaymath}
  fulfilling the following properties:
  \begin{enumerate}
  \item (Differential equation) the equality
    \begin{math}
      \dd_{\mathcal{D}}\widetilde \varphi(\ov A)=0
    \end{math}
    holds;
  \item (Functoriality) for each morphism of smooth complex varieties
    $h\colon X''\to X'$ with $g\circ h $ smooth, we have
    \begin{math}
      h^{\ast} \widetilde \varphi(\ov A)=\widetilde \varphi(h^{\ast}\ov A)
    \end{math};
  \item (Normalization) if $\ov A$ is orthogonally split, then $\widetilde
    \varphi(\ov A)=0$.
  \end{enumerate}
  Then $\widetilde \varphi=0$.
\end{lemma}
\begin{proof}
  The argument of the proof of \cite[Thm.
  2.3]{BurgosLitcanu:SingularBC} applies \emph{mutatis mutandis} to
  the present situation.
\end{proof}

\begin{definition} \label{def:24}
 An \emph{additive genus in Deligne cohomology} is a
characteristic class $\varphi$ for vector bundles of any rank in the sense of
\cite[Def. 1.5]{BurgosLitcanu:SingularBC} that satisfies the
equation
\begin{equation}
  \label{eq:75}
  \varphi(E_{1}\oplus E_{2})=\varphi(E_{1})+\varphi(E_{2}).
\end{equation}

\end{definition}
Let $\DD$ denote the base ring for Deligne cohomology (see
\cite{BurgosLitcanu:SingularBC} before Definition 1.5). A consequence
of \cite[Thm. 1.8]{BurgosLitcanu:SingularBC} is that there is a
bijection between the set of additive genera in Deligne cohomology and
the set of power series in one variable $\DD[[x]]$. To each power
series $\varphi \in \DD[[x]]$ it corresponds the unique additive genus
such that
\begin{displaymath}
  \varphi(L)=\varphi (c_{1}(L))
\end{displaymath}
for every line bundle $L$.

\begin{definition} \label{def:25}
  A \emph{real additive genus} is an additive genus such that the
  corresponding power series belong to $\RR[[x]]$.
\end{definition}

\begin{remark}\label{rem:7}
  It is clear that, if $\varphi$ is a real additive genus, then for each
  vector bundle $E$ we have
  \begin{displaymath}
    \varphi(E)\in \bigoplus_{p}H_{\mathcal{D}}^{2p}(X,\RR(p))
  \end{displaymath}

\end{remark}

We now focus on additive genera, for instance the Chern character is a
real additive genus. Let
$\varphi $ be such a genus. Using Chern-Weil theory, to each hermitian
vector bundle
$\overline E$
on $X$ we can attach a closed characteristic form
\begin{displaymath}
  \varphi (\overline E)\in \bigoplus_{n,p}\mathcal{D}^{n}(X,p).
\end{displaymath}
If $\ov E$ is an object of
$\oV(X)$, then
we define
$$\varphi (\overline E)=\sum_{i}
(-1)^{i} \varphi (\overline E^{i}).$$
If $\overline E$ is acyclic, following \cite[Sec.
2]{BurgosLitcanu:SingularBC},
we associate to it a Bott-Chern characteristic class
\begin{displaymath}
  \widetilde \varphi (\overline E)\in
  \bigoplus_{n,p}\widetilde {\mathcal{D}}^{n-1}(X,p)
\end{displaymath}
that satisfies the differential equation
\begin{displaymath}
  \dd_{\mathcal{D}}\widetilde \varphi (\overline E)=\varphi
  (\overline E).
\end{displaymath}

In fact, \cite[Thm. 2.3]{BurgosLitcanu:SingularBC} for additive
genera can be restated as follows.

\begin{proposition} \label{prop:9}
  Let $\varphi $ be an additive genus. Then there
  is a unique group homomorphism
  \begin{displaymath}
    \widetilde \varphi \colon \KA(X)\to \bigoplus_{n,p}\widetilde
    {\mathcal{D}}^{n-1}(X,p)
  \end{displaymath}
  satisfying the properties:
  \begin{enumerate}
  \item (Differential equation) $\dd_{\mathcal{D}}
    \widetilde\varphi(\ov E)= \varphi(\ov E).$
  \item (Functoriality) If $f\colon X\to Y$ is a morphism of smooth
    complex varieties, then
    $
    \widetilde{\varphi}(\Ld f^{\ast}(\ov{E}))=f^{\ast}(\widetilde{\varphi}(\ov{E})).
    $
  \end{enumerate}
\end{proposition}
\begin{proof}
  For the uniqueness, we observe that, if $\widetilde \varphi$ is a group
  homomorphism then $\widetilde\varphi (\ov 0)=0$. Hence, if $\ov E$
  is a orthogonally split complex, then it is meager and therefore
  $\widetilde\varphi (\ov E)=0$. Thus, the assignment that, to each acyclic
  complex bounded $\ov E$, associates the class $\widetilde
  \varphi([\ov E])$ satisfies the
  conditions of \cite[Thm. 2.3]{BurgosLitcanu:SingularBC}, hence is
  unique. For the existence, we note that
  Bott-Chern classes for additive genera satisfy the
  hypothesis of Theorem \ref{thm:8}. Hence the result follows.
\end{proof}

\begin{remark} \label{rem:6}
If
\begin{displaymath}
  \overline{\varepsilon }:\quad
  0\to \overline{\mathcal{F}}_{m}\to \dots \to
  \overline{\mathcal{F}}_{l}\to 0
\end{displaymath}
is an acyclic complex of coherent sheaves on $X$ provided with
hermitian structures
$\ov {\mathcal{F}}_{i}=(\mathcal{F}_{i},\ov F_{i}\dra
\mathcal{F}_{i})$, by Definition \ref{def:1} we have an object $[\ov
\varepsilon ]\in \KA(X)$, hence a class $\widetilde \varphi([\ov
\varepsilon ])$. In the case of the Chern character, in \cite[Thm.
2.24]{BurgosLitcanu:SingularBC}, a class $\widetilde
{\ch}(\ov \varepsilon )$ is defined. It follows from \cite[Thm.
2.24]{BurgosLitcanu:SingularBC} that both classes agree. That is,
$\widetilde{\ch}([\ov \varepsilon ])=\widetilde{\ch}(\ov \varepsilon
)$. For this reason we will denote $\widetilde \varphi([\ov
\varepsilon ])$ by $\widetilde \varphi(\ov
\varepsilon)$.
\end{remark}

\begin{definition}\label{definition:forms_complexes}
  Let
  $\overline{\mathcal{F}}=(\ov E\overset{\sim}{\dashrightarrow}\mathcal{F})$
  be an object of $\oDb(X)$. Let $\varphi$  denote
  an additive genus. We denote the form
  \begin{equation*}
    \varphi(\overline{\mathcal{F}})=\varphi(\ov E)\in \bigoplus_{n,p}\mathcal{D}^{n}(X,p)
  \end{equation*}
  and the class
  \begin{equation*}
    \varphi(\mathcal{F})=[\varphi(\ov E)]\in
    \bigoplus_{n,p}H_{\mathcal{D}}^{n}(X,\RR(p)).
  \end{equation*}
  Note that the form $\varphi(\overline{\mathcal{F}})$ only depends on the hermitian structure
  and not on a particular representative thanks to Proposition
  \ref{prop:7} and Proposition \ref{prop:9}. The class
  $\varphi(\mathcal{F})$ only depends on the object $\mathcal{F}$ and
  not on the hermitian structure.
\end{definition}
\begin{remark}
  The reason to restrict to additive genera when working with the
  derived category is now clear: there is no
  canonical way to attach a rank to $\oplus_{i\even}\mathcal{F}^{i}$
  (respectively $\oplus_{i\odd} \mathcal{F}^{i}$). The naive choice
  $\rk(\oplus_{i\even} E^{i})$ (respectively $\rk(\oplus_{i\odd} E^{i})$)
  does depend on $E\dashrightarrow\mathcal{F}$. Thus we can not define
  Bott-Chern classes by the general rule from
  \cite{BurgosLitcanu:SingularBC}. The case of a multiplicative genus
  such as the Todd genus will be considered later.
\end{remark}

Next we will construct Bott-Chern classes for isomorphisms
in $\oDb(X)$.
\begin{definition}
  Let $f\colon \overline{\mathcal{F}}\dashrightarrow\overline{\mathcal{G}}$
  be a morphism in $\oDb(X)$ and $\varphi$ an additive
  genus. We define the differential form
  \begin{equation*}
    \varphi(f)=\varphi(\overline{\mathcal{G}})-
    \varphi(\overline{\mathcal{F}}).
  \end{equation*}
\end{definition}

\begin{theorem}\label{theorem:ch_tilde_qiso}
  Let $\varphi$ be an additive genus. There is a unique way to attach to
  every isomorphism in $\oDb(X)$
  \begin{math}
    f\colon (\overline{F}\dashrightarrow\mathcal{F})
    \overset{\sim}{\dashrightarrow}
    (\overline{G}\dashrightarrow\mathcal{G})
  \end{math}
  a Bott-Chern class
  \begin{displaymath}
    \widetilde{\varphi}(f)\in\bigoplus_{n,p}\widetilde{\mathcal{D}}^{n-1}(X,p)
  \end{displaymath}
  such that the following axioms are satisfied:
  \begin{enumerate}
  \item (Differential equation)
    \begin{math}
      \dd_{\mathcal{D}}\widetilde{\varphi}(f)=\varphi(f).
    \end{math}
  \item (Functoriality) If $g\colon X'\rightarrow X$ is a morphism of smooth
    noetherian schemes over $\CC$, then
    \begin{displaymath}
      \widetilde{\varphi}(\Ld g^{\ast}(f))=g^{\ast}(\widetilde{\varphi}(f)).
    \end{displaymath}
  \item (Normalization) If $f$ is a tight isomorphism, then
    \begin{math}
      \widetilde{\varphi}(f)=0.
    \end{math}
  \end{enumerate}
\end{theorem}
\begin{proof}
For the existence we define
\begin{equation}
  \label{eq:56}
  \widetilde \varphi(f)=\widetilde \varphi([f]),
\end{equation}
where $[f]\in\KA(X)$ is the class of $f$ given by equation
\eqref{eq:57}.
That $\widetilde \varphi$ satisfies the axioms
follows from Proposition \ref{prop:9} and Theorem
\ref{thm:9}.

We now focus on the uniqueness. Assume such a theory
$f\mapsto\widetilde{\varphi}_{0}(f)$ exists. Fix $f$ as in the statement.
Since $\widetilde \varphi_{0}$ is well defined, by replacing $\ov F$
by one that is tightly related, we may assume that $f$ is
realized by a morphism of complexes
\begin{displaymath}
  f\colon \ov F\longrightarrow \ov G.
\end{displaymath}
We factorize $f$ as
\begin{displaymath}
  \ov F\overset{\alpha }{\longrightarrow }
  \ov G\oplus \ocone(\ov F,\ov G)[-1]
  \overset{\beta }{\longrightarrow}
  \ov G,
\end{displaymath}
where both arrows are zero on the second factor of the middle
complex. Since $\alpha $ is a tight morphism and $\ocone(\ov F,\ov
G)[-1]$ is acyclic, we are reduced to the case when $\ov F=\ov
G\oplus \ov A$, with $\ov A$ an acyclic complex and $f$ is the
projection onto the first factor.

For each smooth
morphism $g\colon X'\to X$ and each acyclic complex of vector bundles
$\ov E$ on $X'$, we denote
\begin{displaymath}
 \widetilde \varphi_{1}(\ov E)=\widetilde \varphi_{0} (g^{\ast}\ov G\oplus \ov
E\to g^{\ast}\ov G)+\widetilde \varphi(\ov E),
\end{displaymath}
where $\widetilde \varphi$ is the usual Bott-Chern  form for acyclic
complexes of hermitian vector bundles associated to
$\varphi$.
Then $\widetilde \varphi_{1}$ satisfies the
hypothesis of Lemma \ref{lemm:1}, so $\widetilde \varphi_{1}=0$. Therefore
\begin{math}
  \widetilde \varphi(f)=-\widetilde \varphi(\ov A).
\end{math}
\end{proof}

\begin{proposition}\label{proposition:ch_tilde_comp}
Let
$f\colon \overline{\mathcal{F}}\dashrightarrow\overline{\mathcal{G}}$
and $g\colon \overline{\mathcal{G}}\dashrightarrow
\overline{\mathcal{H}}$ be two isomorphisms in $\oDb(X)$. Then:
\begin{displaymath}
	\widetilde{\varphi}(g\circ f)=
        \widetilde{\varphi}(g)+\widetilde{\varphi}(f).
\end{displaymath}
In particular, $\widetilde{\varphi}(f^{-1})=-\widetilde{\varphi}(f)$.
\end{proposition}
\begin{proof}
  The statement follows from Theorem \ref{thm:10}~\ref{item:16}.
\end{proof}

The Bott-Chern classes behave well under shift.

\begin{proposition}
  Let $f\colon \ov{\mathcal{F}}\dra \ov{\mathcal{G}}$ be an
  isomorphism in $\oDb(X)$. Let $f[i]\colon \ov{\mathcal{F}}[i]\dra
  \ov{\mathcal{G}}[i]$ be the shifted isomorphism. Then
  \begin{displaymath}
    (-1)^{i}\widetilde{\varphi}(f[i])=\widetilde{\varphi}(f).
  \end{displaymath}
\end{proposition}
\begin{proof}
  The assignment $f\mapsto (-1)^{i}\widetilde{\varphi}(f[i])$
  satisfies the characterizing properties of Theorem
  \ref{theorem:ch_tilde_qiso}. Hence it agrees with $\widetilde
  \varphi$.
\end{proof}

The following notation will be sometimes used.
\begin{notation}
Let $\mathcal{F}$ be an object of $\Db(X)$ and consider two choices of
hermitian structures $\ov{\mathcal{F}}$ and $\ov{\mathcal{F}}'$. Then
we write
\begin{displaymath}
  \widetilde{\varphi}(\ov{\mathcal{F}},\ov{\mathcal{F}}')=
  \widetilde{\varphi}(\ov{\mathcal{F}}
  \overset{\Id}{\dra}\ov{\mathcal{F}}').
\end{displaymath}
Thus $\dd_{\mathcal{D}}\widetilde{\varphi}
(\ov{\mathcal{F}},\ov{\mathcal{F}}') =
\varphi(\ov{\mathcal{F}}')-\varphi(\ov{\mathcal{F}}).
$
\end{notation}

\begin{example}\label{exm:1}
  Let $\ov
  {\mathcal{F}}=(\mathcal{F},\mathcal{F}\dashrightarrow\overline E)$
  be an object of $\oDb(X)$. Let $\mathcal{H}^{i}$ denote the
  cohomology sheaves of $\mathcal{F}$ and assume that we have chosen
  hermitian structures $\ov{\mathcal{H}}^{i} $ of each
  $\mathcal{H}^{i}$. In the case when the sheaves $\mathcal{H}^{i}$
  are vector bundles and the hermitian structures are hermitian
  metrics, X. Ma, in the paper \cite{Ma:MR1765553}, has associated to
  these data a Bott-Chern class, that we
  denote $M(\ov {\mathcal{F}},\ov {\mathcal{H}})$. By
  the characterization given by Ma of $M(\ov {\mathcal{F}},\ov
  {\mathcal{H}})$, it is immediate that
  \begin{displaymath}
    M(\ov {\mathcal{F}},\ov {\mathcal{H}})=
    \cht(\ov {\mathcal{F}},(\mathcal{F},\ov {\mathcal{H}})),
  \end{displaymath}
  where $(\mathcal{F},\ov {\mathcal{H}})$ is as in Definition
  \ref{def:her_coh}.
\end{example}

Our next aim is to construct Bott-Chern classes for
distinguished triangles.

\begin{definition}
Let $\ov \tau $  be a distinguished triangle in $\oDb(X)$,
\begin{displaymath}
	\overline{\tau}\colon\
        \overline{\mathcal{F}}\overset{u}{\dashrightarrow}\overline{\mathcal{G}}
	\overset{v}{\dashrightarrow}\overline{\mathcal{H}}
        \overset{w}{\dashrightarrow}\overline{\mathcal{F}}[1]
	\overset{u}{\dashrightarrow}\dots
\end{displaymath}
For an additive genus
$\varphi$, we attach the differential form
\begin{displaymath}
	\varphi(\overline{\tau})=\varphi(\overline{\mathcal{F}})-
        \varphi(\overline{\mathcal{G}})+\varphi(\overline{\mathcal{H}}).
\end{displaymath}
\end{definition}

Notice that if $\ov \tau$ is tightly distinguished, then $\varphi(\ov
\tau )=0$. Moreover, for any distinguished triangle $\ov \tau $ as above, the
rotated triangle
\begin{displaymath}
	\overline{\tau}'\colon\
        \overline{\mathcal{G}}\overset{v}{\dashrightarrow}\overline{\mathcal{H}}
	\overset{w}{\dashrightarrow}\overline{\mathcal{F}}[1]
        \overset{-u[1]}{\dashrightarrow}\overline{\mathcal{G}}[1]
	\overset{v[1]}{\dashrightarrow}\dots
\end{displaymath}
satisfies
\begin{math}
  \varphi(\ov \tau ')=-\varphi(\ov \tau).
\end{math}

\begin{theorem}\label{thm:16}
  Let $\varphi$ be an additive genus. There is a unique way to attach to
  every distinguished triangle in $\oDb(X)$
  \begin{displaymath}
    \overline{\tau}:\quad
    \overline{\mathcal{F}}\overset{u}{\dashrightarrow}\overline{\mathcal{G}}
    \overset{v}{\dashrightarrow}\overline{\mathcal{H}}
    \overset{w}{\dashrightarrow}\overline{\mathcal{F}}[1]
    \overset{u[1]}{\dashrightarrow}\dots
  \end{displaymath}
  a Bott-Chern class
  \begin{displaymath}
    \widetilde{\varphi}(\overline{\tau})
    \in\bigoplus_{n,p}\widetilde{\mathcal{D}}^{n-1}(X,p)
  \end{displaymath}
  such that the following axioms are satisfied:
  \begin{enumerate}
  \item (Differential equation)
    \begin{math}
      \dd_{\mathcal{D}}\widetilde{\varphi}(\overline{\tau})=
      \varphi(\overline{\tau}).
    \end{math}
  \item (Functoriality) If $g\colon X^{\prime}\rightarrow X$ is a morphism of smooth noetherian schemes over $\CC$, then
    \begin{displaymath}
      \widetilde{\varphi}(\Ld g^{\ast}(\overline{\tau}))=g^{\ast}\widetilde{\varphi}(\overline{\tau}).
    \end{displaymath}
  \item (Normalization) If $\overline{\tau}$ is tightly distinguished, then
    \begin{math}
      \widetilde{\varphi}(\overline{\tau})=0.
    \end{math}
  \end{enumerate}
\end{theorem}
\begin{proof}
To show the existence we write
\begin{equation}
  \label{eq:58}
 \widetilde \varphi(\ov \tau)=\widetilde \varphi([\ov \tau ]).
\end{equation}
Theorem \ref{thm:10} implies that it satisfies the axioms.

To prove the uniqueness, observe that, by replacing representatives of
the hermitian structures by tightly related ones, we may assume that
the distinguished triangle is represented by
\begin{displaymath}
  \ov F \longrightarrow \ov G \longrightarrow \ocone(\ov F,\ov
  G)\oplus \ov K \longrightarrow \ov F[1],
\end{displaymath}
with $\ov K$ acyclic.
Then Lemma \ref{lemm:1} shows that the axioms imply
\begin{math}
  \widetilde \varphi(\ov \tau )=\widetilde \varphi(\ov K).
\end{math}
\end{proof}

\begin{remark}\label{rem:1} The normalization axiom can be replaced by
  the apparently weaker condition that $\widetilde \varphi(\ov \tau
  )=0$ for all distinguished triangles of the form
  \begin{displaymath}
    \ov {\mathcal{F}}\dra \ov {\mathcal{F}}\overset{\perp}{\oplus}
    \ov {\mathcal{G}}\dra  \ov {\mathcal{G}} \dra
  \end{displaymath}
  where the maps are the natural inclusion and projection.
\end{remark}

Theorem
\ref{thm:10}~\ref{item:17}-\ref{item:49} can be easily translated to Bott-Chern
classes.

\section{Multiplicative genera, the Todd genus and the category
  $\oSm_{\ast/\CC}$}
\label{sec:multiplicative-genera}
Let $\psi $ be a multiplicative genus, such that the piece of degree
zero is $\psi ^{0}=1$, and
\begin{displaymath}
  \varphi=\log(\psi ).
\end{displaymath}
It is a well defined additive genus, because, by the condition above,
the power series $\log (\psi )$ contains only finitely many terms in each degree.

If $\overline \theta $ is either a hermitian vector bundle, a complex
of hermitian vector bundles, a morphism in $\oDb(X)$ or a
distinguished triangle in $\oDb(X)$ we can write
\begin{displaymath}
  \psi (\ov \theta )=\exp(\varphi(\ov \theta )).
\end{displaymath}

All the results of the previous sections can be translated to the
multiplicative genus $\psi $. In particular, if $\ov \theta $ is an
acyclic complex of hermitian vector bundles, an isomorphism in
$\oDb(X)$ or a distinguished triangle in $\oDb(X)$, we define a
Bott-Chern class
\begin{displaymath}
  \widetilde \psi_{m} (\ov \theta)=
  \frac{\exp(\varphi(\ov \theta))-1}{\varphi(\ov \theta)}
  \widetilde\varphi(\ov \theta).
\end{displaymath}

\begin{theorem}\label{thm:11} The characteristic class $\widetilde
  \psi_{m} (\ov \theta)$ satisfies:
  \begin{enumerate}
  \item (Differential equation)
    \begin{math}
      \dd_{\mathcal{D}}\widetilde{\psi}_{m}(\overline{\theta
      })=\psi(\overline{\theta })-1.
    \end{math}
  \item (Functoriality) If $g\colon X^{\prime}\rightarrow X$ is a morphism
    of smooth noetherian schemes over $\CC$, then
    \begin{displaymath}
      \widetilde{\psi}_{m}(\Ld g^{\ast}(\overline{\theta }))
      =g^{\ast}\widetilde{\psi}_{m}(\overline{\theta }).
    \end{displaymath}
  \item (Normalization) If $\overline{\theta }$ is either a meager
    complex, a tight isomorphism or a tightly distinguished triangle,
    then
    \begin{math}
      \widetilde{\psi}_{m}(\overline{\theta })=0.
    \end{math}
  \end{enumerate}
  Moreover $\widetilde \psi_{m} $ is uniquely characterized by these
  properties.
\end{theorem}

\begin{remark}
  For an acyclic complex of vector bundles $\ov E$, using the general
  procedure for arbitrary symmetric power series, we can associate a
  Bott-Chern class $\widetilde \psi (\ov E)$ (see for instance
  \cite[Thm. 2.3]{BurgosLitcanu:SingularBC}) that satisfies the
  differential equation
  \begin{displaymath}
    \dd_{\mathcal{D}} \widetilde \psi (\ov E)= \prod _{k \text{ even }}\psi (\ov
    E^{k})-\prod _{k \text{ odd }}\psi (\ov
    E^{k}),
  \end{displaymath}
  whereas $\widetilde \psi _{m}$ satisfies the differential equation
  \begin{equation} \label{eq:62}
    \dd_{\mathcal{D}} \widetilde \psi _{m}(\ov E)= \prod _{k}\psi (\ov
    E^{k})^{(-1)^{k}}-1.
  \end{equation}
  In fact both Bott-Chern classes are related by
  \begin{equation}\label{eq:76}
    \widetilde \psi _{m}(\ov E)=
    \widetilde \psi (\ov E)\prod_{k \text{ odd}}\psi (\ov E^{k})^{-1}.
  \end{equation}
\end{remark}

The main example of a multiplicative genus with the above properties
is the Todd genus $\Td$. From now on we will treat only this
case. Following the above procedure, to  the Todd genus we can
associate two Bott-Chern classes for acyclic complexes
of vector bundles: the one given by the
general theory, denoted by $\widetilde {\Td}$, and the one given by the
theory of multiplicative genera, denoted $\widetilde {\Td}_{m}$. Both
are related by the equation \eqref{eq:76}. Note however that, for
isomorphisms and distinguished triangles in $\oDb(X)$, we only
have the multiplicative version.

We now consider morphisms between smooth complex varieties and
relative hermitian structures.

\begin{definition}
Let $f\colon X\rightarrow Y$ be a morphism of smooth complex varieties. The
\textit{tangent complex} of $f$ is the complex
\begin{displaymath}
  T_{f}\colon \quad 0\longrightarrow T_{X}
  \overset{df}{\longrightarrow} f^{\ast}T_{Y}\longrightarrow 0
\end{displaymath}
where $T_{X}$ is placed in degree 0 and $f^{\ast}T_{Y}$ is placed in
degree 1. It defines an object $T_{f}\in \Ob \Db(X)$. A
\emph{relative hermitian structure on} $f$ is the choice of an object $\ov T_{f}\in
\oDb(X)$ over $T_{f}$.
\end{definition}

The
following particular situations are of special interest:

\begin{enumerate}
\item[--] suppose $f\colon X\hookrightarrow Y$ is a closed immersion. Let
  $N_{X/Y}[-1]$ be the normal bundle to $X$ in $Y$, considered as a
  complex concentrated in degree 1. By definition, there is a natural
  quasi-isomorphism $p\colon T_{f}\overset{\sim}{\rightarrow} N_{X/Y}[-1]$ in
  $\Cb(X)$, and hence an isomorphism $p^{-1}\colon N_{X/Y}[-1]\overset{\sim}
  {\dashrightarrow} T_{f}$ in $\Db(X)$. Therefore, a hermitian metric $h$
  on the vector bundle $N_{X/Y}$ naturally induces a hermitian
  structure $p^{-1}\colon (N_{X/Y}[-1],h) \dashrightarrow T_{f}$ on $T_{f}$. Let
  $\overline{T}_{f}$ be the corresponding object in $\oDb(X)$. Then we have
  \begin{displaymath}
    \Td(\overline{T}_{f})=\Td(N_{X/Y}[-1],h)=\Td(N_{X/Y},h)^{-1};
  \end{displaymath}
\item[--] suppose $f\colon X\rightarrow Y$ is a smooth morphism. Let
  $T_{X/Y}$ be the relative tangent bundle on $X$, considered as a
  complex concentrated in degree $0$. By definition, there is a
  natural quasi-isomorphism $\iota\colon T_{X/Y}\overset{\sim}{\rightarrow}
  T_{f}$ in $\Cb(X)$. Any choice of hermitian metric $h$ on $T_{X/Y}$
  naturally induces a hermitian structure
  $\iota\colon (T_{X/Y},h)\dashrightarrow T_{f}$. If $\overline{T}_{f}$ denotes
  the corresponding object in $\oDb(X)$, then we find
  \begin{displaymath}
    \Td(\overline{T}_{f})=\Td(T_{X/Y},h).
  \end{displaymath}
\end{enumerate}
Let now $g\colon Y\rightarrow Z$ be another morphism of smooth
varieties over $\CC$. The tangent complexes $T_{f}$, $T_{g}$ and
$T_{g\circ f}$ fit into a distinguished triangle in $\Db(X)$
\begin{displaymath}
	\mathcal{T}\colon T_f\dra T_{g\circ f}\dra
        \Ld f^{\ast}T_g\dra T_f[1].
\end{displaymath}

\begin{definition}\label{def:16}
We denote $\oSm_{\ast/\CC}$ the following data:

(i) The class $\rm{Ob}\,\oSm_{\ast/\CC}$ of smooth complex varieties.

(ii) For each $X,Y\in \rm{Ob}\,\oSm_{\ast/\CC}$, a set of morphisms $\oSm_{\ast/\CC}(X,Y)$ whose elements are
pairs $\ov f=(f,\ov T_{f})$, where $f\colon X\to Y$ is a projective morphism and $\ov T_{f}$
is a hermitian structure on $T_{f}$. When $\ov f$ is given we will
denote the hermitian structure by $T_{\ov f}$. A hermitian
structure on $T_{f}$ will also be called a hermitian structure on
$f$.

(iii) For each pair of morphisms $\ov f\colon X\to Y$
and $\ov g\colon Y\to Z$, the composition defined as
\begin{displaymath}
    \ov g\circ \ov f=(g\circ f,\ocone(\Ld f^{\ast} T_{\ov g}[-1], T_{\ov f})).
\end{displaymath}
\end{definition}

We shall prove (Theorem \ref{thm:17}) that $\oSm_{\ast/\CC}$ is a
category. Before this, we proceed with some examples emphasizing some
properties of the composition rule.

\begin{example}\label{exm:3}
  Let $f\colon X\to Y$
  and $g\colon Y\to Z$, be projective morphisms of smooth complex
  varieties. Assume
  that we have chosen hermitian metrics on the tangent vector bundles
  $T_{X}$, $T_{Y}$ and $T_{Z}$. Denote by $\ov f$, $\ov g$ and
  $\ov{g\circ f}$ the morphism of $\oSm_{\ast/\CC}$  determined by
  these metrics. Then
  \begin{displaymath}
    \ov g\circ \ov f=\ov{g\circ f}.
  \end{displaymath}
  This is seen as follows. By the choice of metrics, there is a tight
  isomorphism
  \begin{displaymath}
 \ocone(T_{\ov f},T_{\ov{g\circ f}})\to \Ld
  f^{\ast}T_{\ov g}.   
  \end{displaymath}
  Then the natural maps
  \begin{displaymath}
    T_{\ov g\circ \ov f}\to \ocone(\Ld
      f^{\ast}T_{\ov g}[-1],T_{\ov f})\to
    \ocone(\ocone(T_{\ov f},T_{\ov{g\circ f}})[-1],T_{\ov f})\to
    T_{\ov{g\circ f}}
  \end{displaymath}
  are tight isomorphisms.
\end{example}

\begin{example} \label{exm:4}
    Let $f\colon X\to Y$
  and $g\colon Y\to Z$, be smooth projective morphisms of smooth complex
  varieties.  Choose hermitian metrics on the relative tangent
  vector bundles
  $T_{f}$, $T_{g}$ and $T_{g\circ f}$. Denote by $\ov f$, $\ov g$ and
  $\ov{g\circ f}$ the morphism of $\oSm_{\ast/\CC}$  determined by
  these metrics. There is a short exact sequence of hermitian vector
  bundles
  \begin{displaymath}
    \ov \varepsilon \colon
    0\longrightarrow \ov T_{f}
    \longrightarrow \ov T_{g\circ f}
    \longrightarrow f^{\ast} \ov T_{g}
    \longrightarrow 0,
  \end{displaymath}
  that we consider as an acyclic complex declaring $f^{\ast} \ov T_{g}$
  of degree $0$. The morphism $f^{\ast}T_{\ov g}[-1]\dra T_{\ov f}$ is
  represented by the diagram
  \begin{displaymath}
    \xymatrix{
      & \ocone(T_{\ov f},T_{\ov{g\circ f}})[-1]
      \ar[dl]_{\sim} \ar[rd] &\\
      f^{\ast}T_{\ov g}[-1] && T_{\ov f}.
    }
  \end{displaymath}
Thus, by the definition of a composition we have
\begin{displaymath}
  T_{\ov g \circ \ov f}=
  \ocone(\ocone(T_{\ov f},T_{\ov{g\circ f}})[-1],f^{\ast}T_{\ov
    g}[-1])[1]\oplus
  \ocone(\ocone(T_{\ov f},T_{\ov{g\circ f}})[-1],T_{\ov f}).
\end{displaymath}
In general this hermitian structure is different to $T_{\ov{g\circ
    f}}$.

\emph{Claim.} The equality of hermitian structures
\begin{equation}
  \label{eq:84}
  T_{\ov g \circ \ov f}=
  T_{\ov{g\circ f}}+ [\ov \varepsilon ]
\end{equation}
holds.
\begin{proof}[Proof of the claim]
  We have a commutative diagram of distinguished triangles
  \begin{displaymath}
    \xymatrix{
      \ov{\varepsilon} &T_{\ov{f}}\ar[r]\ar[d]_{\Id} &T_{\ov{g\circ f}}\ar[r]\ar[d]
      &f^{\ast}T_{\ov{g}}\ar[d]_{\Id}\ar@{-->}[r] &T_{\ov{f}}[1]\ar[d]_{\Id}\\
      \ov{\tau} &T_{\ov{f}}\ar[r] &T_{\ov{g}\circ \ov{f}}\ar[r]
      &f^{\ast}T_{\ov{g}}\ar@{-->}[r] &T_{\ov{f}}[1].
      }
      \end{displaymath}
By construction the triangle $\ov{\tau}$ is tightly distinguished,
hence $[\ov{\tau}]=0$. Therefore, according to Theorem \ref{thm:10}
\ref{item:48}, we have
\begin{displaymath}
  [T_{\ov{g\circ f}}\rightarrow T_{\ov{g}\circ\ov{f}}]=[\ov{\varepsilon}].
\end{displaymath}
The claim follows.
\end{proof}
\end{example}

\begin{theorem}\label{thm:17}
$\oSm_{\ast/\CC}$ is a category.
\end{theorem}

\noindent\emph{Proof}
The only non-trivial fact to prove is the associativity of the composition, given by the
following lemma:
\begin{lemma}\label{lem:Sm_cat}
  Let $\ov{f}:X\to Y$, $\ov{g}:Y\to Z$ and $\ov{h}:Z\to W$ be
  projective morphisms together with hermitian structures. Then
  $\ov{h}\circ(\ov{g}\circ\ov{f})=(\ov{h}\circ\ov{g})\circ\ov{f}$.
\end{lemma}
\begin{proof}
  First of all we observe that if the hermitian structures on
  $\ov{f}$, $\ov{g}$ and $\ov{h}$ come from fixed hermitian metrics on
  $T_{X}$, $T_{Y}$, $T_{Z}$ and $T_{W}$, Example \ref{exm:3} ensures
  that the proposition holds. For the general case, it is enough to
  see that if the proposition holds for a fixed choice of hermitian
  structures $\ov{f}$, $\ov{g}$, $\ov{h}$, and we change the metric on
  $f$, $g$ or $h$, then the proposition holds for the new choice of
  metrics. We treat, for instance, the case when we change the
  hermitian structure on $g$, the proof of the other cases being analogous. 
  Denote by $\ov{g}'$ the new hermitian structure on
  $g$. Then there exists a unique class $\varepsilon\in\KA(Y)$ such
  that $T_{\ov{g}'}=T_{\ov{g}}+\varepsilon$. According to the
  definitions, we have
  \begin{displaymath}
    T_{\ov{h}\circ(\ov{g}'\circ\ov{f})}=
    \ocone((g\circ
    f)^{\ast}T_{\ov{h}}[-1],\ocone(f^{\ast}(T_{\ov{g}}+\varepsilon)[-1],T_{\ov{f}}))
    =T_{\ov{h}\circ(\ov{g}\circ\ov{f})}+f^{\ast}\varepsilon.
  \end{displaymath}
  Similarly, we find
  \begin{displaymath}
    T_{(\ov{h}\circ\ov{g}')\circ\ov{f}}=
    \ocone(f^{\ast}\ocone(g^{\ast}T_{\ov{h}}[-1],T_{\ov{g}})[-1]+
    f^{\ast}(-\varepsilon),
    T_{\ov{f}})
    =T_{(\ov{h}\circ\ov{g})\circ\ov{f}}+f^{\ast}\varepsilon.
  \end{displaymath}
  By assumption,
  $T_{\ov{h}\circ(\ov{g}\circ\ov{f})}=T_{(\ov{h}\circ\ov{g})\circ\ov{f}}$. Hence
  the relations above show
  \begin{displaymath}
    T_{\ov{h}\circ(\ov{g}'\circ\ov{f})}=T_{(\ov{h}\circ\ov{g}')\circ\ov{f}}.
  \end{displaymath}
  This concludes the proofs of Lemma \ref{lem:Sm_cat} and of Theorem \ref{thm:17}.
\end{proof}

Let $f\colon X\to Y$ and $g\colon Y\to Z$ be projective morphisms of
smooth complex varieties. By the definition of composition, hermitian
structures on $f$ and $g$ determine a hermitian structure on $g\circ
f$. Conversely we have the following result.

\begin{lemma}\label{lemm:4}
  Let $\ov {g}$ and $\ov {g\circ f}$ be hermitian structures on $g$
  and $g\circ f$. Then there is a unique hermitian structure $\ov f$ on
  $f$ such that
  \begin{equation}
    \label{eq:85}
    \ov{g\circ f}=\ov g\circ \ov f.
  \end{equation}
\end{lemma}
\begin{proof}
  We have the distinguished triangle
  \begin{displaymath}
    T_{f}\dra T_{g\circ f}\dra f^{\ast}T_{g}\dra T_{f}[1].
  \end{displaymath}
  The unique hermitian structure that satisfies equation
  \eqref{eq:85} is $\ocone(T_{\ov{g\circ
      f}},f^{\ast}T_{\ov{g}})[-1]$.
\end{proof}

\begin{remark}
  By contrast with the preceding result, it is not true in general that
  hermitian structures $\ov f$ and $\ov{g\circ f}$ determine a
  unique hermitian structure $\ov g$ that satisfies equation
  \eqref{eq:85}. For instance, if $X=\emptyset$, then any hermitian
  structure on $g$ will satisfy this equation.
\end{remark}

If $\Sm_{\ast/\CC}$ denotes the category of smooth complex varieties and
projective morphisms and $\mathfrak{F}\colon \oSm_{\ast/\CC}\to
\Sm_{\ast/\CC}$ is the forgetful functor, for any object $X$ we have
that
\begin{align*}
  \Ob \mathfrak{F}^{-1}(X)&=\{X\},\\
  \Hom_{\mathfrak{F}^{-1}(X)}(X,X)&=\KA(X).
\end{align*}

To any arrow $\ov f\colon X\to Y$ in $\oSm_{\ast/\CC}$ we associate a
Todd form
\begin{equation}\label{eq:1}
  \Td(\ov f):=\Td(T_{\ov f})\in \bigoplus_{p}\mathcal{D}^{2p}(X,p).
\end{equation}

The following simple properties of $\Td(\ov f)$ follow directly from
the definitions.

\begin{proposition}
  \begin{enumerate}
  \item Let $\ov f\colon X\to Y$ and $\ov{g}\colon Y\to Z$ be
    morphisms in $\oSm_{\ast/\CC}$. Then
    \begin{displaymath}
      \Td(\ov g\circ\ov f)=f^{\ast}\Td(\ov{g})\bullet\Td(\ov{f}).
    \end{displaymath}
  \item Let $f,f'\colon X\to Y$ be two morphisms in $\oSm_{\ast/\CC}$
    with the same underlying algebraic morphism. There is an
    isomorphism $\ov \theta \colon T_{\ov f}\to T_{\ov f'}$ whose
    Bott-Chern class $\widetilde{\Td}_{m}(\ov\theta)$ satisfies
    \begin{displaymath}
      \dd_{\mathcal{D}}\widetilde {\Td}_{m}(\ov \theta )=
      \Td(T_{\ov f'}) \Td(T_{\ov f})^{-1}-1.
    \end{displaymath}
  \end{enumerate}
\end{proposition}


\newcommand{\etalchar}[1]{$^{#1}$}
\newcommand{\noopsort}[1]{} \newcommand{\printfirst}[2]{#1}
  \newcommand{\singleletter}[1]{#1} \newcommand{\switchargs}[2]{#2#1}
\providecommand{\bysame}{\leavevmode\hbox to3em{\hrulefill}\thinspace}
\providecommand{\MR}{\relax\ifhmode\unskip\space\fi MR }
\providecommand{\MRhref}[2]{%
  \href{http://www.ams.org/mathscinet-getitem?mr=#1}{#2}
}
\providecommand{\href}[2]{#2}

\end{document}